\documentclass[12pt,reqno]{amsart}
\usepackage{amsmath,amstext,amssymb,amscd}
\usepackage{verbatim}
\usepackage{enumerate}
\usepackage{mathrsfs}
\usepackage[dvipsnames,usenames]{xcolor}

\usepackage{amsthm}

\newtheorem{theorem}{Theorem}[section]
\newtheorem{lemma}[theorem]{Lemma}
\newtheorem{corollary}[theorem]{Corollary}
\newtheorem{proposition}[theorem]{Proposition}
\newtheorem{assumption}[theorem]{Assumption}

\theoremstyle{definition}
\newtheorem{definition}[theorem]{Definition}

\theoremstyle{remark}
\newtheorem{remark}[theorem]{Remark}

\numberwithin{equation}{section}

\begin{document}

\def\a{\alpha}
\def\b{\beta}
\def\d{\delta}
\def\g{\gamma}
\def\l{\lambda}
\def\o{\omega}
\def\s{\sigma}
\def\t{\tau}
\def\th{\theta}
\def\r{\rho}
\def\D{\Delta}
\def\G{\Gamma}
\def\O{\Omega}
\def\e{\epsilon}
\def\p{\phi}
\def\P{\Phi}
\def\S{\Psi}
\def\E{\eta}
\def\m{\mu}
\def\grad{\nabla}
\def\bar{\overline}
\newcommand{\A}{\mathcal{A}}
\newcommand{\reals}{\mathbb{R}}
\newcommand{\naturals}{\mathbb{N}}
\newcommand{\ints}{\mathbb{Z}}
\newcommand{\complex}{\mathbb{C}}
\newcommand{\rationals}{\mathbb{Q}}
\newcommand{\innerprod}[1]{\left\langle#1\right\rangle}
\newcommand{\dualprod}[1]{\left\langle#1\right\rangle}
\newcommand{\norm}[1]{\left\|#1\right\|}
\newcommand{\abs}[1]{\left|#1\right|}
\newcommand{\sn}{\mathop{\mathrm{sn}}}

\author[Y. Guo]{Yanqiu Guo}
\address{Department of Computer Science and Applied Mathematics, Weizmann Institute of Science\\
Rehovot 76100, Israel} \email{yanqiu.guo@weizmann.ac.il}

\author[E. S. Titi]{Edriss S. Titi}
\address{Department of Mathematics and Department of Mechanical and Aerospace Engineering\\
University of California, Irvine, California 92697-3875, USA. Also Department of Computer
Science and Applied Mathematics, Weizmann Institute of Science, Rehovot 76100, Israel} \email{etiti@math.uci.edu and edriss.titi@weizmann.ac.il}

\title[Persistency of Analyticity for nonlinear wave equations]
{Persistency of analyticity for nonlinear wave equations: an energy-like approach}

\begin{abstract}
We study the persistence of the Gevrey class regularity of solutions to nonlinear wave equations with real analytic nonlinearity. Specifically, it is proven that the solution remains in a Gevrey class, with respect to some of its spatial variables,  during its whole life-span, provided the initial data is from the same Gevrey class with respect to these spatial variables. In addition, for the special Gevrey class of analytic functions, we find a lower bound for the radius of the spatial analyticity of the solution that might shrink either algebraically or exponentially, in time, depending on the structure of the nonlinearity. The  standard $L^2$ theory for the Gevrey class regularity is employed; we also employ energy-like methods for a generalized version of Gevrey classes based on the $\ell^1$ norm of Fourier transforms (Wiener algebra). After careful comparisons, we observe an indication that the $\ell^1$ approach provides a better lower bound for the radius of analyticity of the solutions than the $L^2$ approach. We present our results in the case of period boundary conditions, however, by employing exactly the same tools and proofs one can obtain similar results for the nonlinear wave equations and the nonlinear Schr\"odinger equation, with real analytic nonlinearity,  in certain domains and manifolds without physical boundaries, such as  the whole space  $\mathbb{R}^n$, or on the sphere $\mathbb{S}^{n-1}$.
\end{abstract}

\date{August 12, 2013. {\bf To appear in:} {\it Bulletin  of Institute of Mathematics, Academia Sinica (N.S.)}}
\maketitle


\maketitle

\bigskip
{\it Dedicated to Professor Neil Trudinger on the occasion of his  70th birthday.}
\bigskip
\section{Introduction}
In this article we investigate the persistence of the Gevrey class regularity of solutions to the nonlinear wave equation
\begin{align} \label{wave-eqn}
\begin{cases}
\Box u+u+f(t, x, u, \nabla u, u_t)=0 \text{\;\;in\;\;} {\mathbb{T}^n}=[0,2\pi]^n;\\
u(0)=u_0, \;\;u_t(0)=u_1,
\end{cases}
\end{align}
with periodic boundary condition on $u$, where $f$ is periodic with respect to the spatial variable $x$.

Concerning the analytic regularity for (\ref{wave-eqn}), we shall mention a few results in the literature. If $f$ is real analytic in all its arguments and the initial data are real analytic, then the classical Cauchy-Kowalewski Theorem asserts a unique real analytic solution of (\ref{wave-eqn}) for $t$ near 0. Ovsiannikov \cite{Ovsjannikov-71} and Nirenberg \cite{Nirenberg-72} generalized this theorem to the case when the nonlinearity $f$ is merely continuous in $t$ with values as an analytic function of the other variables. It is important to note that by these findings we only know that the solution to (\ref{wave-eqn}) is analytic in a small neighborhood of $t=0$. Later, Alinhac and Metivier \cite{Alinhac-Metivier-84} improved this result by showing that the analyticity of the solution to (\ref{wave-eqn}) lasts for as long as a classical solutions exists. Recently, Kuksin and Nadirashvili \cite{Kuksin-Nadirashvili-12} demonstrated a short and more transparent proof of this property (and actually more general) based on the nonlinear semigroup generated by (\ref{wave-eqn}). Nevertheless, none of the above mentioned results provide a lower estimate for the radius of analyticity of the solution to the nonlinear wave equation (\ref{wave-eqn}). We attempt to answer this question in the present paper by employing the well-developed Gevrey class theory based on energy estimate tools.

Gevrey classes were introduced by Maurice Gevrey (1918) to generalize real analytic functions. Briefly speaking, a Gevrey class is an intermediate space between the spaces of $C^{\infty}$ functions and real-analytic functions. The tools and results developed in this paper are concerned with Gevrey class functions in domains or manifolds without boundaries; specifically, either periodic boundary conditions, the whole space or the sphere. In literature, the Gevrey class energy-like technique has been a powerful tool for studying the regularity of solutions to nonlinear evolution equations, such as  Navier-Stokes equations \cite{Biswas-Swanson, Cao-Rammaha-Titi2, Foias-Temam-89}, parabolic PDE's \cite{Cao-Rammaha-Titi1, Ferrari-Titi-98, Kreiss}, and Euler equations \cite{Kukavica-Vicol-09, Levermore-Oliver-97}; see also \cite{Bardos-Benachour} for persistence of analyticity of solutions of  Euler equations in domains with physical boundaries.

The authors of \cite{Larios-Titi-10} investigated the analytic regularity of Euler-Voigt equations, and provided rigorous justification to the formal tools and proofs that were introduced in \cite{Levermore-Oliver-97} (see also \cite{Kalantatov-Levant-Titi} for the analytic regularity of the attractor of the Navier-Stoke-Voigt mode). Since the Euler-Voigt system behaves like a hyperbolic system of equations, we are able to adopt the techniques of \cite{Larios-Titi-10} to study the nonlinear wave equation (\ref{wave-eqn}). Also we draw ideas from \cite{Ferrari-Titi-98}, in which the authors  establish the Gevrey class regularity of analytic solutions for general nonlinear parabolic equations with analytic nonlinearity. Moreover, the authors of \cite{Kukavica-Vicol-09, Levermore-Oliver-97, Oliver-Titi-00, Oliver-Titi-01} provide additional  tools for estimating lower bounds of the radius of analyticity of solutions to evolution differential equations.

A function $u\in C^{\infty}(\mathbb T^d)$ is said to be of \emph{Gevrey class s}, for some $s\geq 1$, if there exist constants $\rho>0$ and $M<\infty$, such that for every $x\in \mathbb T^d$ and every $\alpha \in \mathbb N^d$ one has
$$|\partial^{\alpha}u(x)|\leq M\left(\frac{\alpha !}{\rho^{|\alpha|}}\right)^s.$$
Next we shall introduce the Gevrey-Sobolev classes which will be used in this paper. We set $A:=\sqrt{I-\Delta}$, thus  $H^p({\mathbb{T}^n})=\mathcal D(A^p)$ is the Sobolev space of functions with the periodic boundary condition. The norm of $H^p$ is given by
\begin{align*}
\norm{u}_{H^p({\mathbb{T}^n})}=\left(\sum_{j\in \ints^n}|u_j|^2(1+|j|^2)^p \right)^{\frac{1}{2}},
\end{align*}
where $u_j$ are the Fourier coefficients of $u$, i.e., $u=\sum_{j\in \ints^n} u_j e^{i j\cdot x}$.
A \emph{Gevrey-Sobolev class} of order $s\geq 1$ is defined by $\mathcal D(A^p e^{\t A^{1/s}})$, with its norm
\begin{align} \label{G-norm}
\norm{A^p e^{\t A^{1/s}} u}=\left(\sum_{j\in \ints^n}|u_{j}|^2(1+|j|^2)^p e^{2\t(1+|j|^2)^{\frac{1}{2s}}}\right)^{\frac{1}{2}}, \;\;s\geq 1,
\end{align}
where $\t>0$. Throughout, $\norm{\cdot}$ denotes the $L^2$ norm.
It is known that $\mathcal D(A^p e^{\t A^{1/s}})$ is a subclass of the Gevrey class $s$ (cf. \cite{Levermore-Oliver-97}). More importantly, for the special case $s=1$, the Gevrey-Sobolev class $\mathcal D(A^{p} e^{\t A})$, $\t>0$, corresponds to the set of real analytic functions with radius of analyticity bounded below by $\t$.

For the sake of clarity, throughout the paper we focus on functions in the Gevrey class $s=1$, i.e., the space of real analytic functions. Nonetheless, all results in this manuscript equally hold  for any Gevrey class $s\geq 1$.
Indeed, the main idea, as we will see later, relies on the fact that
\begin{align*}
e^{\tau(1+|m+j|^2)^{\frac{1}{2s}}}\leq e^{\tau(1+|m|^2)^{\frac{1}{2s}}} \cdot
e^{\tau(1+|j|^2)^{\frac{1}{2s}}}, \text{\;\;for\;\;} s\geq 1.
\end{align*}

It is important to stress that we are interested in initial data which are analytic in merely some of its spatial variables, and aim to show the persistence of analyticity with respect to these variables. To this end, we shall introduce a slightly modified Gevrey-Sobolev class.
  Indeed, we define the operator $\mathscr A^p: \mathcal D(\mathscr A^p)\subset L^2({\mathbb{T}^n})\rightarrow L^2({\mathbb{T}^n})$ by $\mathscr A^p u=\sum_{j\in \ints^n} u_j |j'|^p  e^{i j \cdot x}$ with its domain given by
\begin{align*}
\mathcal D(\mathscr A^p)=\Big\{u\in L^2({\mathbb{T}^n}): u=\sum_{j\in \ints^n} u_j e^{i j \cdot x},\;
\sum_{j\in \ints^n}|u_j|^2 |j'|^{2p}<\infty\Big\},
\end{align*}
where $j'$ represents the first $m$ components of $j=(j_1,\ldots,j_n)\in \ints^n$, that is, $j'=(j_1,\ldots,j_m)$, for some fixed integer $m\in [1,n]$.
With the operators $A$ and $\mathscr A$ we introduce a new Gevrey-Sobolev class $\mathcal D(A^p e^{\t \mathscr A})$, with its norm
\begin{align} \label{G-norm'}
\norm{A^p e^{\t \mathscr A}u}=\left(\sum_{j\in \ints^n}|u_{j}|^2(1+|j|^2)^p
e^{2\t|j'|}\right)^{\frac{1}{2}}.
\end{align}
In fact, $\mathcal D(A^p e^{\t \mathscr A})$ is a Hilbert space. Also Corollary \ref{cor-1} in the appendix states that if $u\in \mathcal D(A^p e^{\t \mathscr A})$, $p>\frac{n}{2}$, then $u$ is real analytic in its first $m$ arguments $(x_1,\ldots,x_m)$ and $\t$ is a lower bound of the radius of analyticity with respect to these variables.

The paper is organized as follows: in Section \ref{sec-L-2} we use the Gevrey-Sobolev class $\mathcal D(A^p e^{\t \mathscr A})$ to analyse the regularity of solutions to nonlinear wave equations and estimate the radius of analyticity. In Section \ref{sec-l-1} we work on the same problem by applying another type of Gevrey classes based on the $\ell^1$ norm of Fourier transforms (Wiener algebra). Finally, in Section \ref{sec-compare} we compare these two different estimates via the careful investigation of a nonlinear Klein-Gordon equation and conclude that the Wiener algebra approach provides a ``better" lower bound for the radius of analyticity of the solutions than the $L^2$ approach (see also \cite{Oliver-Titi-01} and \cite{Guo-Titi-13} for this kind of comparison).

\bigskip
\section{$L^2$ estimate} \label{sec-L-2}
In this section we employ the Gevrey-Sobolev class $\mathcal D(A^p e^{\t \mathscr A})$ to study the analytic regularity of solutions to (\ref{wave-eqn}). Since the norm (\ref{G-norm'}) of $\mathcal D(A^p e^{\t \mathscr A})$ is based on the $L^2$ norm, the estimate in this section relies on the standard $L^2$ theory.

We begin with the definition of a solution of (\ref{wave-eqn}). Throughout the paper, $(\cdot,\cdot)$ denotes the $L^2$ inner product; while $\langle \cdot, \cdot \rangle$ denotes the duality pairing between $(H^1)'$ and $H^1$.
\begin{definition} \label{solution}
We say a function $u\in C([0,T];H^p)$, with $u_t\in C([0,T];H^{p-1})$ and
$u_{tt}\in C([0,T];H^{p-2})$, $p\geq 1$, is a \emph{solution} of the initial value problem (\ref{wave-eqn}) provided $u(0)=u_0 \in H^p$, $u_t(0)=u_1\in H^{p-1}$ and
\begin{align}
\langle u_{tt},\phi \rangle+(A u,A\phi)+\langle f(t,x,u,\nabla u,u_t),\phi \rangle=0
\end{align}
for every $\phi \in H^1$, and for every $t\in [0,T]$.
\end{definition}

First of all let us deal with a simpler nonlinearity $f(t,x,u)$, i.e., we shall study the equation
\begin{align} \label{wave-eqn-1}
\begin{cases}
\Box u+u+f(t, x, u)=0;\\
u(0)=u_0, \;\;u_t(0)=u_1,
\end{cases}
\end{align}
where $f$ satisfies the following assumption:
\begin{assumption} \label{ass-1}
Let $f(t,x,u)=\sum_{j\in \ints^n} \hat c_j(t,u) e^{i j\cdot x}$, where $\hat c_j(t,u)=\sum_{k=0}^{\infty}a_{jk}(t)u^k$, where $a_{jk}(t)$ are continuous functions on $[0,T]$. Suppose $f$ has a majorising function
\begin{align} \label{def-g}
g(t,s):=\sum_{k=0}^{\infty} \sum_{j\in \ints^n} |a_{jk}(t)| (1+|j|^2)^{\frac{p}{2}} e^{\lambda |j'|}s^k
\end{align}
converging for all $s\in \reals$, $t\geq 0$, where $\lambda>0$, and $j'=(j_1,\dots,j_m)$, $1\leq m \leq n$.
\end{assumption}

\begin{remark} \label{remark-1}
Note Assumption \ref{ass-1} implies the following properties of $f$: it is continuous in $t$, real analytic in $(x_1,\ldots,x_m)\in \mathbb T^m$, and real analytic (entire) in $u\in \reals$.
In fact, for any fixed $t\geq 0$ and $u\in \reals$, $f(t,x,u)$ is, as a function of $x$, in the Gevrey class $\mathcal D(A^p e^{\lambda \mathscr A})$. To see this, we note
\begin{align} \label{e-17}
&\sum_{j\in \ints^n}|\hat c_j(t,u)| (1+|j|^2)^{\frac{p}{2}} e^{\lambda |j'|}\notag\\
&\leq \sum_{k=0}^{\infty} \sum_{j\in \ints^n} |a_{jk}(t)| (1+|j|^2)^{\frac{p}{2}} e^{\lambda |j'|} |u|^k
=g(t,|u|)<\infty,
\end{align}
which implies
\begin{align*}
\sum_{j\in \ints^n}|\hat c_j(t,u)|^2 (1+|j|^2)^p e^{2\lambda |j'|}<\infty,
\end{align*}
that is to say $f(t,x,u)\in \mathcal D(A^p e^{\lambda \mathscr A})$, for any fixed $t\geq 0$ and $u\in \reals$.
On the other hand, for every fixed $t\geq 0$ and $x\in {\mathbb{T}^n}$, we notice
$$f(t,x,u)=\sum_{j\in \ints^n} \hat c_j(t,u) e^{i j\cdot x}
=\sum_{k=0}^{\infty}\left(\sum_{j\in \ints^n}e^{i j\cdot x} a_{jk}(t) \right) u^k$$ converges absolutely for all $u\in \reals$, that is to say,
$f(t,x,u)$ is real analytic (entire) with respect to the variable $u\in \reals$.

In particular, any nonlinear function $\tilde f(t,u)$ (independent of $x$), which is continuous in $t$ and real analytic (entire) with respect to $u$, satisfies Assumption \ref{ass-1}.
\end{remark}

Now we state our first result.

\begin{theorem} \label{thm-1}
Suppose the nonlinearity $f(t,x,u)$ satisfy Assumption \ref{ass-1}. Let $u_0\in \mathcal D(A^{p+1}e^{\sigma \mathscr A})$ and $u_1\in \mathcal D(A^p e^{\sigma \mathscr A})$, where $p>\frac{n}{2}$ and $\sigma>0$. Assume the initial-value problem (\ref{wave-eqn-1}) has a unique solution $u\in C([0,T];H^p)$ with $u_t\in C([0,T];H^{p-1})$, in the sense of Definition \ref{solution}.
Suppose $\t(t)$ is the solution of the differential equation
\begin{align} \label{radius-1}
\t'(t)=-\t^3(t) h(t) \text{\;\;with\;\;} \t(0)=\t_0:=\min\{\sigma,\lambda\},
\end{align}
that is,
\begin{align} \label{sol-tau}
\t(t)=\left(2\int_0^t h(s)ds+\t_0^{-2}\right)^{-\frac{1}{2}},
\end{align}
where the function $h(t)\geq 0$ for all $t\in [0,T]$, given in (\ref{def-h}), below, depends on $\norm{u(t)}_{H^p}$, $\norm{A^{p+1}e^{\t_0 \mathscr A}u_0}$
and $\norm{A^p e^{\t_0 \mathscr A}u_1}$. Then, $u(t)\in \mathcal D(A^{p+1}e^{\t(t)\mathscr A})$ and $u_t(t)\in \mathcal D(A^p e^{\t(t) \mathscr A})$, for all $t\in [0,T]$.
\end{theorem}

\begin{remark}
Theorem \ref{thm-1} shows the solution of (\ref{wave-eqn-1}) remains spatial analytic, with respect to some of its spatial variables, during its whole life-span, provided the initial data is analytic with respect to these spatial variables.

Notice that the parameters $\sigma$ and $\lambda$ represent, respectively, the radius of analyticity of the initial data and the nonlinearity $f(t,x,u)$ for fixed $t$ and $u$. (\ref{radius-1}) shows the analytic regularity of the solution is influenced by both $\sigma$ and $\lambda$. In Assumption \ref{ass-1} we suppose $\lambda$ is a constant; nevertheless, we can instead assume $\l$ depends on $t$, and under this more general setting, one may derive $\t(t)\leq \l(t)$ for all $t\in [0,T]$.

Also it is clear from Theorem \ref{thm-1} that, if $g(t,s)$, the majorising function of the nonlinearity $f$, is algebraic, and $\norm{u(t)}_{H^p}$ is bounded above by some algebraic increasing function, then the radius of analyticity of the solution $u(t)$ shrinks at most algebraically fast. To demonstrate this idea,  we will consider later, as an example,  the three-dimensional defocusing Klein-Gordon equation (cubic nonlinear wave equation),  (\ref{cubic-wave}), for which it is known that the $H^2-$norm remains bounded. In particular, we will obtain, by applying our method to this explicit example,  a lower bound for  the radius of analyticity, $\tau(t)$, of the solution that behaves like  $ \frac{1}{t}$, for large $t$. Hence the radius of analyticity in this case cannot decay faster than algebraic.
\end{remark}

\begin{proof}
The proof draws ideas from \cite{Ferrari-Titi-98, Larios-Titi-10, Levermore-Oliver-97, Oliver-Titi-00}, all of which are based on the tools developed in   \cite{Foias-Temam-89} (see also \cite{Kreiss}). We establish our result by employing the Galerkin method. Denote by $P_N$ the $L^2$ projection onto the span of
$\{e^{i j\cdot x}\}_{|j|\leq N}$. Let ${u_N}(t)=\sum_{|j|\leq N}u_{N,j}(t) e^{i j \cdot x}$ be the solutions of the Galerkin system associated with the initial-value problem (\ref{wave-eqn}), i.e., ${u_N}$ satisfies
\begin{align} \label{Galerkin}
\Box u_N+u_N+P_N f(t,x,{u_N})=0
\end{align}
with the initial condition ${u_N}(0)=P_N u_0$ and $u_N'(0)=P_N u_1$.
Clearly, (\ref{Galerkin}) generates a second-order $(2N+1)$-dimensional  system of ordinary differential equations with continuous nonlinearity in the unknown functions $u_{N,j}(t)$. By the Cauchy-Peano Theorem, for every $N\geq 1$ system (\ref{Galerkin}) has a solution ${u_N}(t)$ on $[0,T_N]$ with $u_{N,j}(t)\in C^2[0,T_N]$, for $|j|\leq N$.

In what follows we focus our analysis on the interval $[0,T_N]$.
In order to derive an a-priori estimate for $u_N$,
we follow the standard approach in \cite{Ferrari-Titi-98, Foias-Temam-89, Larios-Titi-10}. Applying the operator $A^p e^{\t \mathscr A}$ to both sides of (\ref{Galerkin}) and taking the $L^2$ inner product with $A^p e^{\t \mathscr A} u_N'$, one has
\begin{align} \label{e-1}
(A^p e^{\t \mathscr A}u_N'', \;&A^{p} e^{\t \mathscr A}u_N')+(A^p e^{\t \mathscr A} A^2{u_N}, A^pe^{\t \mathscr A}u_N') \notag\\
&+(A^pe^{\t \mathscr A}P_N f(t,x,{u_N}), A^pe^{\t \mathscr A}u_N')=0.
\end{align}
It is clear that
\begin{align} \label{e-2}
&\frac{1}{2}\frac{d}{dt}\norm{A^pe^{\t \mathscr A}u_N'}^2 \notag\\
&=(A^pe^{\t \mathscr A}u_N'',A^pe^{\t \mathscr A}u_N')+\t'(t)\norm{A^p\mathscr A^{\frac{1}{2}}e^{\t \mathscr A}u_N'}^2,
\end{align}
and
\begin{align} \label{e-3}
&\frac{1}{2}\frac{d}{dt}\norm{A^{p+1}e^{\t \mathscr A}{u_N}}^2 \notag\\
&=(A^{p+1}e^{\t \mathscr A}{u_N}, A^{p+1}e^{\t \mathscr A}u_N')+\t'(t)\norm{A^{p+1}\mathscr A ^{\frac{1}{2}}e^{\t \mathscr A}{u_N}}^2.
\end{align}
Substituting (\ref{e-2}) and (\ref{e-3}) into (\ref{e-1}) gives
\begin{align} \label{e-6}
&\frac{1}{2}\frac{d}{dt}\left(\norm{A^pe^{\t \mathscr A}u_N'}^2+\norm{A^{p+1}e^{\t \mathscr A}{u_N}}^2\right)\notag\\
&\leq \t'(t)\left(\norm{A^{p} \mathscr A^{\frac{1}{2}} e^{\t \mathscr A}u_N'}^2
+\norm{A^{p+1} \mathscr A^{\frac{1}{2}}e^{\t \mathscr A}{u_N}}^2\right) \notag\\
&+\norm{A^pe^{\t \mathscr A}P_N f(t,x,{u_N})}\norm{A^pe^{\t \mathscr A}u_N'}.
\end{align}

Now we compute the Gevrey norm of $f(t,x,u_N)$. Recall from Assumption \ref{ass-1},
we express $f(t,x,u)=\sum_{j\in \ints^n} \hat c_j(t,u) e^{i j\cdot x}$, where $\hat c_j(t,u)=\sum_{k=0}^{\infty}a_{jk}(t)u^k$. Also, Lemma \ref{lem-2} in the appendix shows that $\mathcal D(A^p e^{\t \mathscr A})$ is a Banach algebra, for $p>\frac{n}{2}$, so
\begin{align} \label{e-18}
&\norm{A^p e^{\t \mathscr A} f(t,x,{u_N(x)})} \notag\\
&\leq \sum_{j\in \ints^n} \norm{A^p e^{\t \mathscr A} \left[\hat c_j(t,u_N(x))e^{ij\cdot x}\right]}\notag\\
&\leq \sum_{j\in \ints^n} C_0
\norm{A^p e^{\t \mathscr A} \hat c_j(t,{u_N(x)})}
\norm{A^p e^{\t \mathscr A} e^{ij\cdot x}}.
\end{align}
Let us mention Lemma 2 in \cite{Ferrari-Titi-98}, which states if $u\in \mathcal D(A^p e^{\t A})$ and $F$
be an analytic function with a majorising function $g_0$, then $F(u)\in \mathcal D(A^p e^{\t A})$ and
$\norm{A^p e^{\t A} u}\leq (1+C_p^{-1})g_0(C_p \norm{A^p e^{\t A}u})$.
By analogy with this result we derive
\begin{align} \label{e-4}
\norm{A^p e^{\t \mathscr A} \hat c_j(t,u)}\leq
C_1 \sum_{k=0}^{\infty} |a_{jk}(t)|C_0^{k-1}\norm{A^p e^{\t \mathscr A}u}^k.
\end{align}
Clearly,
\begin{align} \label{e-19}
\norm{A^p e^{\t \mathscr A}e^{ij\cdot x}}=(1+|j|^2)^{\frac{p}{2}} e^{\t|j'|}.
\end{align}
If we require $\t(t)\leq \lambda$, for all $t\in [0,T_N]$, then combining (\ref{e-18})-(\ref{e-19}) yields
\begin{align} \label{e-5}
&\norm{A^p e^{\t \mathscr A} f(t,x,{u_N})} \notag\\
&\leq C_1 \sum_{j\in \ints^n}
\sum_{k=0}^{\infty} |a_{jk}(t)|C_0^k \norm{{A^p e^{\t \mathscr A}u_N}}^k   (1+|j|^2)^{\frac{p}{2}} e^{\lambda |j'|}\notag\\
&=C_1 g(t,C_0\norm{A^p e^{\t \mathscr A}u_N})<\infty,
\end{align}
due to Assumption \ref{ass-1}.

We proceed to estimate the nonlinear term $g(t,C_0\norm{A^p e^{\t \mathscr A}u_N})$.
Similar to Lemma 8 in \cite{Oliver-Titi-00}, by using the elementary inequality $e^{2x}\leq e^2+x^{\ell} e^{2x}$ for all $x\geq 0$, $\ell>0$, we deduce (taking $\ell=3$)
\begin{align} \label{e-8}
&\norm{A^p e^{\t \mathscr A}u_N}^2
=\sum_{|j|\leq N}|u_{N,j}|^2(1+|j|^2)^p e^{2\t|j'|} \notag\\
&\leq e^2\sum_{|j|\leq N}|u_{N,j}|^2(1+|j|^2)^p
+\t^3\sum_{|j|\leq N}|u_{N,j}|^2(1+|j|^2)^p |j'|^3 e^{2\t|j'|}\notag\\
&= e^2\norm{{u_N}}_{H^p}^2+
\t^3\norm{A^p \mathscr A^{\frac{3}{2}} e^{\t \mathscr A}u_N}^2.
\end{align}

Inspired by (\ref{e-8}), we intend to obtain a similar estimate for $\norm{A^p e^{\t \mathscr A}u_N}^k$ for any integer $k\geq 1$. In fact,
\begin{align} \label{e-8'}
&\norm{A^p e^{\t \mathscr A}u_N}^k=\Big(\sum_{|j|\leq N}|u_{N,j}|^2(1+|j|^2)^p e^{2\t|j'|}\Big)^{\frac{k}{2}} \notag\\
&\leq \Big(e^2\sum_{|j|\leq N}|u_{N,j}|^2(1+|j|^2)^p
+\t^{\frac{6}{k}}\sum_{|j|\leq N}|u_{N,j}|^2(1+|j|^2)^p |j'|^{\frac{6}{k}} e^{2\t|j'|}\Big)^{\frac{k}{2}}.
\end{align}
By the discrete H\"older's inequality it follows
\begin{align} \label{Holder}
&\sum_{|j|\leq N}|u_{N,j}|^2(1+|j|^2)^p |j'|^{\frac{6}{k}} e^{2\t|j'|}\notag\\
&\leq \Big(\sum_{|j|\leq N}|u_{N,j}|^2(1+|j|^2)^{p}e^{2\t|j'|}\Big)^{\frac{k-2}{k}}
\Big(\sum_{|j|\leq N}|u_{N,j}|^2(1+|j|^2)^p |j'|^3 e^{2\t|j'|}\Big)^{\frac{2}{k}}\notag\\
&\leq \norm{A^p e^{\t \mathscr A}u_N}^{\frac{2(k-2)}{k}}\norm{A^p \mathscr A^{\frac{3}{2}} e^{\t \mathscr A}{u_N}}^{\frac{4}{k}}.
\end{align}
A combination of (\ref{e-8'}) and (\ref{Holder}) yields
\begin{align} \label{e-9}
\norm{A^p e^{\t \mathscr A}{u_N}}^k
\leq 2^{\frac{k-2}{2}}\Big(e^k\norm{{u_N}}^k_{H^p}+\t^3\norm{A^p e^{\t \mathscr A}{u_N}}^{k-2}
\norm{A^p \mathscr A^{\frac{3}{2}} e^{\t \mathscr A}{u_N}}^2\Big)
\end{align}
for all $k\geq 2$.

Notice that (\ref{e-9}) is not valid for $k=1$. To deal with the case $k=1$, we simply let $k=2$ in (\ref{e-9}) followed by taking the square root, obtaining
\begin{align} \label{e-9'}
\norm{A^p e^{\t \mathscr A}u_N}&\leq e \norm{{u_N}}_{H^p}+\t^{\frac{3}{2}}\norm{A^p \mathscr A^{\frac{3}{2}}e^{\t \mathscr A}{u_N}}\notag\\
&\leq e\norm{{u_N}}_{H^p}+\frac{1}{2}\t^3\norm{A^p \mathscr A^{\frac{3}{2}}e^{\t \mathscr A}{u_N}}^2+\frac{1}{2}
\end{align}
where Cauchy's inequality has been used.

Applying estimates (\ref{e-9}) and (\ref{e-9'}) together with the definition (\ref{def-g}) of $g$ shows
\begin{align} \label{e-10}
g(t,C_0\norm{A^p e^{\t \mathscr A}u_N})&\leq g(t,C_0(\sqrt{2}e\norm{{u_N}}_{H^p}+1))\notag\\
&+\t^3 g(t,C_0(\sqrt 2 \norm{A^p e^{\t \mathscr A}u_N}+1))
\norm{A^p \mathscr A^{\frac{3}{2}} e^{\t \mathscr A}{u_N}}^2.
\end{align}

Obviously, $\norm{A^p \mathscr A^{\frac{3}{2}} e^{\t \mathscr A}{u_N}}
\leq \norm{A^{p+1} \mathscr A^{\frac{1}{2}}e^{\t \mathscr A}{u_N}}$.
Thus, combining (\ref{e-6}), (\ref{e-5}) and (\ref{e-10}) yields
\begin{align} \label{e-11}
&\frac{1}{2}\frac{d}{dt}\left(\norm{A^p e^{\t \mathscr A}u_N'}^2+\norm{A^{p+1}e^{\t \mathscr A}{u_N}}^2\right) \notag\\
&\leq C_1 g(t,C_0(\sqrt{2}e\norm{{u_N}}_{H^p}+1))\norm{A^p e^{\t \mathscr A}u_N'} \notag\\
&+\Big[\t'+\t^3 C_1 g(t,C_0(\sqrt 2 \norm{A^p e^{\t \mathscr A}u_N}+1))\norm{A^p e^{\t \mathscr A}u_N'}\Big]\notag\\
&\hspace{1 in} \times \Big(\norm{A^p \mathscr A^{\frac{1}{2}}e^{\t \mathscr A}u_N'}^2+\norm{A^{p+1} \mathscr A^{\frac{1}{2}}e^{\t \mathscr A}{u_N}}^2\Big)
\end{align}
for all $t\in [0,T_N]$.

In order to use the estimate (\ref{e-11}) to obtain the analytic regularity of the solutions, we need to
pass to the limit as $N\rightarrow \infty$.
Therefore we shall study the convergence of $u_N$ to the solution $u$. Indeed, if we let $\t=0$ in (\ref{e-6}) and (\ref{e-5}), then it follows that
\begin{align} \label{e-30}
&\frac{1}{2}\frac{d}{dt}\left(\norm{u_N'}^2_{H^p}+\norm{u_N}^2_{H^{p+1}}\right) \leq C_1 g(t,C_0 \norm{{u_N}}_{H^p})\norm{u_N'}_{H^p}.
\end{align}
Define $U_N:=(u_N,u_N')$ and $U_N(0)=U_0:=(u_0,u_1)$. Also set $\mathcal H:=H^{p+1}\times H^{p}$.
Then (\ref{e-30}) is reduced to
\begin{align} \label{e-30'}
\frac{d}{dt} \norm{U_N(t)}_{\mathcal H}\leq C_1 g(t,C_0 \norm{{U_N}}_{\mathcal H}).
\end{align}
Integrating (\ref{e-30'}) on $[0,t]\subset [0,T_N]$ gives
\begin{align} \label{e-31}
\norm{U_N(t)}_{\mathcal H}\leq \norm{U_N(0)}_{\mathcal H} + C_1 \int_0^t  g(s,C_0 \norm{U_N(s)}_{\mathcal H}) ds.
\end{align}
 Since $\norm{U_N(0)}_{\mathcal H}\leq \norm{U_0}_{\mathcal H}$, by the continuity of $\norm {U_N(t)}_{\mathcal H}$, there exists a $T_N^*\in [0,T_N]\subset [0,T]$ such that $\norm{U_N(t)}_{\mathcal H}\leq \norm{U_0}_{\mathcal H}+1$, for all $t\in [0,T_N^*]$. Thus, by (\ref{e-31}) it follows that for all $t\in [0,T_N^*]$,
\begin{align} \label{e-32}
&\norm{U_N(t)}_{\mathcal H} \leq
\norm{U_0}_{\mathcal H}+ C_1 t \max_{s\in [0,T]} g\left(s,C_0 (\norm{U_0}_{\mathcal H}+1)\right).
\end{align}
Note the right hand side of (\ref{e-32}) is finite since $g$ is continuous in its arguments.

In order to see that $T_N^*$ does not approach 0 as $N\rightarrow \infty$, we demand the right-hand side of (\ref{e-32}) to be smaller than or equal to $\norm{U_0}_{\mathcal H}+1$, then we have the inequality (\ref{e-32}) holds for all $t\in [0,T^*]$ where
\begin{align} \label{e-33}
T^*=\min\left\{\frac{1}{C_1 \max_{s\in [0,T]} g\left(s,C_0 (\norm{U_0}_{\mathcal H}+1)\right)},\;\;T\right\}.
\end{align}
Thus $0<T^*\leq T_N^*$, for all $N$, and
\begin{align} \label{e-34}
\norm{U_N(t)}_{\mathcal H}\leq \norm{U_0}_{\mathcal H}+1 \text{\;\;on\;\;} [0,T^*]
\end{align}
for all $N$. Moreover, by (\ref{Galerkin}) and (\ref{e-34}) one has $u''_N$ are uniformly bounded in $C([0,T^*];H^{p-1})$, and due to the uniform bound (\ref{e-34}) of the $\mathcal H$-norm, there exist $\tilde U:=(\tilde u,\tilde u')\in \mathcal H$ such that $U_N \rightarrow \tilde U$ weak$-*$ in $L^{\infty}(0,T^*;\mathcal H)$, and $u_N'' \rightarrow \tilde u''$ weak$-*$ in $L^{\infty}(0,T^*;H^{p-1})$. Then, it follows by the Aubin's Compactness Theorem \cite{Simon-87} that on a subsequence $u_N\rightarrow \tilde u$ strongly in $C([0,T^*];H^{p})$.

Next we show that $\tilde u$ is a solution of (\ref{wave-eqn-1}) on $[0,T^*]$. For an arbitrary $\phi\in H^1$, we obtain from (\ref{Galerkin}) that
\begin{align*}
\langle u_N'',\phi \rangle +(A u_N,A\phi)+(P_N f(t,x,u_N),\phi)=0.
\end{align*}
To see the convergence of the nonlinearity, we consider
\begin{align} \label{e-35}
\norm{P_N f(t,x,u_N)-f(t,x,\tilde u)}
&\leq \norm{P_N f(t,x,u_N)-P_N f(t,x,\tilde u)} \notag\\
&+\norm{P_N f(t,x,\tilde u)-f(t,x,\tilde u)}.
\end{align}
Now, we estimate the right-hand side of (\ref{e-35}). Since $H^p \hookrightarrow L^{\infty}$, for $p>\frac{n}{2}$, we obtain from (\ref{e-34}) that there exists $C>0$ such that $\sup_{0\leq t\leq T^*}\norm{u_N(t,x)}_{L^{\infty}}\leq C$, for all $N$ and $\sup_{0\leq t\leq T^*}\norm{u(t,x)}_{L^{\infty}}\leq C$. In addition, since $\sum_{k=0}^{\infty}\sum_{j\in \ints^n}|a_{jk}(t)|s^k$ converges for all $s\in \reals$, it follows that
$\sum_{k=0}^{\infty}\sum_{j\in \ints^n}|a_{jk}(t)|k s^{k-1}$ also converges for all $s\in \reals$. Hence, for all $t\in [0,T^*]$,
\begin{align} \label{e-36}
&\norm{P_N f(t,x,u_N)-P_N f(t,x,\tilde u)} \notag\\
&\leq \left[\int_{{\mathbb{T}^n}}   \left(\sum_{j\in \ints^n}|\hat c_j(t,u_N)-\hat c_j(t,\tilde u)|\right)^2 dx \right]^{\frac{1}{2}} \notag\\
&\leq \norm{u_N-\tilde u} \sum_{j\in \ints^n} \max_{|s|\leq C} \left|\frac{d}{ds}\hat c_j(t,s)\right|\notag\\
&\leq \norm{u_N-\tilde u}\sum_{j\in \ints^n}\sum_{k=0}^{\infty}|a_{jk}(t)|k C^k \longrightarrow 0 \, ,
\end{align}
as $N\rightarrow \infty$, where we have used the Mean Value Theorem. Combining (\ref{e-35}) and (\ref{e-36}) shows
$(P_N f(t,x,u_N),\phi)\rightarrow (f(t,x,\tilde u),\phi)$. Thus $\tilde u$ is a solution.
But by the assumption $u$ is the unique solution, and hence one must have $\tilde u=u$ on $[0,T^*]$. It follows that $u_N\rightarrow u$ strongly in $C([0,T^*];H^p)$.

If $T^*<T$, then we let the time begins at $t=T^*$ and make the extension by reiterating the previous argument. By the formula (\ref{e-33}) of $T^*$, it is clear that after \emph{finite} number of steps, we obtain a sequence ${u_N}\rightarrow u$ strongly in $C([0,T];H^p)$.

So there exists $N'\in \naturals$ such that
$\sqrt{2}e\norm{{u_N}(t)}_{H^p}\leq \sqrt{2}e\norm{u(t)}_{H^p}+1$, on $[0,T]$, if $N\geq N'$. Due to the fact $g(t,s)$ is increasing for every fixed $t$, it follows that
\begin{align} \label{ok-3}
g(t,C_0(\sqrt{2}e\norm{{u_N}(t)}_{H^p}+1))
\leq g(t,C_0(\sqrt{2}e\norm{u(t)}_{H^p}+2)) \, ,
\end{align}
on $[0,T]$, for $N\geq N'$.

Now we return to the Gevrey norm estimate (\ref{e-11}). By the above arguments, we know (\ref{e-11}) is valid on $[0,T]$, and thus, if we let $\t_N(t)$ be the solution of the ODE
\begin{align} \label{ok-6}
\t_N'+\t_N^3 C_1 g(t,C_0(\sqrt 2 \norm{A^p e^{\t_N \mathscr A}u_N}+1))\norm{A^p e^{\t_N \mathscr A}u_N'}=0 \, ,
\end{align}
with $\t_N(0)=\tau_0=\min\{\sigma,\lambda\}$, for all $t\in [0,T]$, then it follows that
\begin{align} \label{ok-1}
&\frac{1}{2}\frac{d}{dt}\left(\norm{A^p e^{\t_N \mathscr A}u_N'}^2+\norm{A^{p+1}e^{\t_N \mathscr A}{u_N}}^2\right) \notag\\
&\leq C_1 g(t,C_0(\sqrt{2}e\norm{{u_N}}_{H^p}+1))\norm{A^p e^{\t_N \mathscr A}u_N'}.
\end{align}
Define
\begin{align} \label{ok-2}
{Y_N}(t):=\Big(\norm{A^p e^{\t_N(t) \mathscr A}u_N'(t)}^2+\norm{A^{p+1}e^{\t_N(t) \mathscr A}{u_N(t)}}^2\Big)^{\frac{1}{2}}.
\end{align}
Then, (\ref{ok-1}) reads
\begin{align*}
Y_N(t)Y_N'(t)\leq C_1 g(t,C_0(\sqrt{2}e\norm{{u_N}}_{H^p}+1)) Y_N(t), \;\;\;\;t\in [0,T],
\end{align*}
and along with (\ref{ok-3}), one has
\begin{align} \label{ok-4}
Y_N(t) \leq Y_0+C_1 \int_0^t  g(s,C_0(\sqrt{2}e\norm{{u(s)}}_{H^p}+2)) ds :=\xi(t)\, ,
\end{align}
for $t\in [0,T]$, $N\geq N'$, where $Y_0=(\norm{A^p e^{\tau_0 \mathscr A}u_1}^2+\norm{A^{p+1}e^{\tau_0 \mathscr A}u_0}^2 )^{\frac{1}{2}}$.

If we let $\t(t)$ satisfy the equation
\begin{align} \label{ok-5}
\t'(t)+C_1 \t^3(t) g(t,C_0(\sqrt 2 \xi(t) +1))\xi(t)=0 \;,\;\;\;\text{with} \; \t(0)=\t_0 \, ,
\end{align}
then by (\ref{ok-6}) and (\ref{ok-4}), we conclude $\t_N(t)\geq \t(t)$, for all $t\in [0,T]$.
It follows that
\begin{align*}
\Big(\norm{A^p e^{\t(t) \mathscr A}u_N'(t)}^2+\norm{A^{p+1}e^{\t(t) \mathscr A}{u_N(t)}}^2\Big)^{\frac{1}{2}}\leq Y_N(t)\leq \xi(t)\, ,
\end{align*}
for all $t\in [0,T]$. In order to write (\ref{ok-5}) in a more compact form, we set
\begin{align} \label{def-h}
h(t):=C_1g(t,C_0(\sqrt 2 \xi(t) +1))\xi(t)\, ,
\end{align}
where $\xi(t)$ is defined in (\ref{ok-4}). Then, the equation (\ref{ok-5}) reads $\t'(t)+\t^3(t) h(t)=0$
with $\t(0)=\t_0$, and its solution is given in (\ref{sol-tau}).

Finally, we shall obtain the analytic regularity of the solution $(u,u_t)$ by passing to the limit $N\rightarrow \infty$. Note
\begin{align*}
\norm{A^{p+1} e^{\t(t) \mathscr A}u_N(t)}^2
=\sum_{|j|\leq N} (1+|j|^2)^{p+1} e^{2 \t(t) |j'|} |u_{N,j}(t)|^2 \leq \xi(t)
\end{align*}
for any $N\geq N'$, $t\in [0,T]$. Thus, for every fixed number $N_0$, for all $t\in [0,T]$,
\begin{align*}
\sum_{|j|\leq N_0} (1+|j|^2)^{p+1} e^{2 \t(t) |j'|} |u_{j}(t)|^2
=\lim_{N\rightarrow \infty}\sum_{|j|\leq N_0} (1+|j|^2)^{p+1} e^{2 \t(t) |j'|} |u_{N,j}(t)|^2  \leq \xi(t).
\end{align*}
Note, in the above formula, we pass to the limit into finite sums and use the fact $u_N(t)\rightarrow u(t)$ in $H^p$, $p>\frac{n}{2}$, for every $t\in [0,T]$.
Therefore, since $N_0\geq 0$ is arbitrarily selected, $\norm{A^{p+1} e^{\t(t) \mathscr A}u(t)}\leq \xi(t)$ for all $t\in [0,T]$.
Similarly, one can show $\norm{A^{p} e^{\t(t) \mathscr A}u'(t)}\leq \xi(t)$ for all $t\in [0,T]$.
\end{proof}

\begin{remark}
In the proof of Theorem \ref{thm-1} we have essentially justified the existence of a solution of the initial value problem (\ref{wave-eqn-1}). The uniqueness of solutions can be obtained by routine arguments.
\end{remark}

\begin{remark}
Define $\tau_0=\min \{\sigma, \lambda\}$. If we set $\t(t)$ to be a constant $\tau_0$ in the inequality (\ref{e-6})
and (\ref{e-5}), then we obtain
$$\frac{d}{dt}y(t) \leq C_1 g(t,y(t))$$
where $y(t)=\sqrt{\norm{A^p e^{\tau_0 \mathscr A}u_N'(t)}^2+\norm{A^{p+1}e^{\tau_0 \mathscr A}{u_N(t)}}^2}$. Since $y(0)$ is finite, by the continuity of $g$, we see that $y(t)$ is finite for a short time $T'$. However $T'$ may be smaller than the life span $T$ of the solution. The bottom line is that the lower bound $\t(t)$ of the radius of analyticity of a solution can remain constant for a short time but need to decrease in order to prevent the blow up of the Gevrey norm.
\end{remark}

\smallskip

Next we consider the equation (\ref{wave-eqn}) with the general nonlinearity $f(t,x,u,\nabla u,u_t)$ which satisfies the following assumption:
\begin{assumption} \label{ass-2}
Let $$f(t,x,u, \nabla u,u_t)=\sum_{j\in \ints^n} \hat c_j(t,u,\nabla u,u_t) e^{i j\cdot x}$$
where $\hat c_j(t,u,\nabla u,u_t):=\sum_{\b}a_{j \b}(t)u^{\b_0}u_{x_1}^{\b_1}\cdots u_{x_n}^{\b_n}u_t^{\b_{n+1}}$, $\b=(\b_0,\b_1,\ldots,\b_n)\in \naturals^{n+2}_0$, $a_{j\b}(t)$ are continuous functions in $t$.
Suppose $f$ has a majorising function
$$g(t,s_0,s_1,\ldots,s_n,s_{n+1}):=\sum_{\b\in \naturals^{n+2}_0} b_{\b}(t) s_0^{\b_0}s_1^{\b_1}\cdots s_n^{\b_n} s_{n+1}^{\b_{n+1}}$$ converges for all $(s_0,s_1,\ldots,s_n,s_{n+1})\in \reals^{n+2}$, $t\geq 0$, where $b_{\b}(t):=\sum_{j\in \ints^n} |a_{j \b}(t)| (1+|j|^2)^{\frac{p}{2}} e^{\lambda |j'|}$,
$\lambda>0$.
\end{assumption}

\begin{remark}
Similar to Remark \ref{remark-1} one can see that Assumption \ref{ass-2} implies $f$ is continuous in $t$, real analytic (of special Gevrey class of regularity) in $(x_1,\ldots,x_m)\in \mathbb T^m$ , and real analytic (entire) with respect to the rest arguments. In particular, any nonlinear function $\tilde f(t,u,\nabla u,u_t)$ (independent of $x$), which is continuous in $t$ and real analytic (entire) in the other variables, satisfies Assumption \ref{ass-2}.
\end{remark}

The following result is concerned with the analytic regularity of solutions to the more general nonlinear wave equation (\ref{wave-eqn}).

\begin{theorem} \label{thm-2}
Let $u_0\in \mathcal D(A^{p+1}e^{\sigma\mathscr A})$ and $u_1\in \mathcal D(A^p e^{\sigma \mathscr A})$ where $p>\frac{n}{2}$ and
$\sigma>0$. Assume the initial-value problem (\ref{wave-eqn}) has a unique solution $u\in C([0,T];H^{p+1})$ with $u_t\in C([0,T];H^p)$, in the sense of Definition \ref{solution}.
Suppose $\t(t)$ is the solution of the differential equation
\begin{align*}
\t'(t)=-\t(t) \eta(t) \text{\;\;with\;\;} \t(0)=\t_0:=\min\{\sigma,\lambda\},
\end{align*}
that is,
\begin{align*}
\t(t)=\t_0 e^{-\int_0^t \eta (s)ds},
\end{align*}
where $\eta(t)\geq 0$, for all $t\in [0,T]$, defined in (\ref{def-eta}) below, depends on $\norm{u(t)}_{H^{p+1}}$, $\norm{u_t(t)}_{H^p}$, $\norm{A^{p+1}e^{\t_0 \mathscr A}u_0}$
and $\norm{A^p e^{\t_0 \mathscr A}u_1}$.
Then, $u(t)\in \mathcal D(A^{p+1}e^{\t(t)\mathscr A})$ and $u_t(t)\in \mathcal D(A^p e^{\t(t) \mathscr A})$, for all $t\in [0,T]$.
\end{theorem}

\begin{remark}
By Theorem \ref{thm-2}, if the majorising function $g$ in Assumption \ref{ass-2} is an algebraic function, and the growth rates of $\norm{u(t)}_{H^{p+1}}$ and $\norm{u_t(t)}_{H^p}$, $p>\frac{n}{2}$, are not higher than algebraic, then the analyticity radius of $u(t)$ shrinks at most exponentially fast as $t\rightarrow \infty$. One may compare this result with Theorem \ref{thm-1} to see how the structure of the nonlinearity affects the lower bound $\t(t)$ of the radius of spatial analyticity. On the other hand, a satisfactory estimate of the radius of analyticity of $u$ depends on sharp estimates of the Sobolev norms $\norm{u(t)}_{H^{p+1}}$ and $\norm{u_t(t)}_{H^p}$, $p>\frac{n}{2}$.
\end{remark}

\begin{proof}
Note the following estimates are formal, which can be justified rigorously using the Galerkin method similar to the proof of Theorem \ref{thm-1}. Suppose $u$ is a solution of (\ref{wave-eqn}). By referring to (\ref{e-6}) we have
\begin{align} \label{eqn-1}
&\frac{1}{2}\frac{d}{dt}\left(\norm{A^pe^{\t \mathscr A}u_t}^2+\norm{A^{p+1}e^{\t \mathscr A}{u}}^2\right)\notag\\
&\leq \t'(t)\left(\norm{A^p \mathscr A^{\frac{1}{2}} e^{\t \mathscr A}u_t}^2+
\norm{A^{p+1} \mathscr A^{\frac{1}{2}} e^{\t \mathscr A}{u}}^2\right) \notag\\
&+\norm{A^pe^{\t \mathscr A}f(t,x,u,\nabla u,u_t)}\norm{A^p e^{\t \mathscr A}u_t}.
\end{align}
We shall evaluate the nonlinear term $\norm{A^pe^{\t \mathscr A}f(t,x,u,\nabla u,u_t)}$.

Like (\ref{e-18}) one has
\begin{align}  \label{e-20}
&\norm{A^p e^{\t \mathscr A} f(t,x,u,\nabla u,u_t)} \notag\\
&\leq \sum_{j\in \ints^n} C_0
\norm{A^p e^{\t \mathscr A} \hat c_j(t,u,\nabla u,u_t)}
\norm{A^p e^{\t \mathscr A} e^{ij\cdot x}}.
\end{align}

Recall $\hat c_j(t,u,\nabla u,u_t):=\sum_{\b}a_{j \b}(t)u^{\b_0}u_{x_1}^{\b_1}\cdots u_{x_n}^{\b_n}u_t^{\b_{n+1}}$, where $\b=(\b_0,\b_1,\ldots,\b_n)\in \naturals^{n+2}_0$. By Lemma \ref{lem-2} we obtain
\begin{align} \label{e-21}
&\norm{A^p e^{\t \mathscr A} \hat c_j(t,u,\nabla u,u_t)} \notag\\
&\leq \tilde C \sum_{\b} |a_{j\b}|C_0^{|\b|-1}\norm{A^p e^{\t \mathscr A}u}^{\b_0}
\prod_{k=1}^n \norm{A^p e^{\t \mathscr A}u_{x_k}}^{\b_k}
\norm{A^p e^{\t \mathscr A}u_{x_t}}^{\b_{n+1}}
\end{align}
where $|\b|=\sum_{k=0}^{n+1}\b_k$.

It follows from (\ref{e-19}), (\ref{e-20}) and (\ref{e-21}) that
\begin{align}  \label{e-22}
&\norm{A^p e^{\t \mathscr A} f(t,x,u,\nabla u,u_t)} \notag\\
&\leq \tilde C g\left(t,C_0\norm{A^p e^{\t \mathscr A}u},C_0\norm{A^p e^{\t \mathscr A}u_{x_1}},
\cdots,C_0\norm{A^p e^{\t \mathscr A}u_{x_n}},C_0\norm{A^p e^{\t \mathscr A}u_{t}}\right).
\end{align}
Next we estimate the right-hand side of (\ref{e-22}).

For $k=1,\ldots,n$ with $\b_k\geq 1$, similar to (\ref{e-8})-(\ref{e-9'}) we compute
\begin{align} \label{e-23}
&\norm{A^p e^{\t \mathscr A}u_{x_k}}^2
\leq \norm{A^{p+1} e^{\t \mathscr A}u}^2
=\sum_{j\in \ints^n}|u_j|^2(1+|j|^2)^{p+1}e^{2\t|j'|}\notag\\
&\leq e^2\sum_{j\in \ints^n}|u_j|^2(1+|j|^2)^{p+1} +\t^{\frac{2}{\b_{k}(n+2)}}\sum_{j\in \ints^n}|u_j|^2
(1+|j|^2)^{p+1}|j'|^{\frac{2}{\b_k(n+2)}}
e^{2\t|j'|}\notag\\
&\leq e^2\norm{u}^2_{H^{p+1}}+\t^{\frac{2}{\b_{k}(n+2)}}
\norm{A^{p+1}e^{\t \mathscr A}u}^{2\left(1-\frac{2}{\b_k(n+2)}\right)}
\norm{A^{p+1} \mathscr A^{\frac{1}{2}} e^{\t \mathscr A}u}^{\frac{4}{\b_k(n+2)}},
\end{align}
where the H\"older's inequality has been used.
Taking the square root on both sides of (\ref{e-23}) immediately gives
\begin{align*}
\norm{A^p e^{\t \mathscr A}u_{x_k}}\leq e\norm{u}_{H^{p+1}}+\t^{\frac{1}{\b_{k}(n+2)}}
\norm{A^{p+1}e^{\t \mathscr A}u}^{1-\frac{2}{\b_k(n+2)}}
\norm{A^{p+1} \mathscr A^{\frac{1}{2}} e^{\t \mathscr A}u}^{\frac{2}{\b_k(n+2)}}.
\end{align*}
Hence
\begin{align} \label{e-24}
&\norm{A^p e^{\t \mathscr A}u_{x_k}}^{\b_k}\notag\\
&\leq 2^{\b_k-1} \left(e^{\b_k}\norm{u}_{H^{p+1}}^{\b_k}+\t^{\frac{1}{n+2}}
\norm{A^{p+1}e^{\t \mathscr A}u}^{\b_k-\frac{2}{n+2}}
\norm{A^{p+1} \mathscr A^{\frac{1}{2}} e^{\t \mathscr A}u}^{\frac{2}{n+2}}\right),
\end{align}
for $\b_k\geq 1$ where $k=1,\ldots,n$.

Similarly, we have
\begin{align} \label{e-25}
&\norm{A^p e^{\t \mathscr A}u}^{\b_0}\notag\\
&\leq 2^{\b_0-1} \left(e^{\b_0}\norm{u}_{H^{p}}^{\b_0}+\t^{\frac{1}{n+2}}
\norm{A^{p}e^{\t \mathscr A}u}^{\b_0-\frac{2}{n+2}}
\norm{A^{p} \mathscr A^{\frac{1}{2}} e^{\t \mathscr A}u}^{\frac{2}{n+2}}\right),
\end{align}
for $\b_0 \geq 1$. Also, for $\b_{n+1}\geq 1$ one has
\begin{align} \label{e-26}
&\norm{A^p e^{\t \mathscr A}u_t}^{\b_{n+1}}\notag\\
&\leq 2^{\b_{n+1}-1} \left(e^{\b_{n+1}}\norm{u_t}_{H^{p}}^{\b_{n+1}}+\t^{\frac{1}{n+2}}
\norm{A^{p}e^{\t \mathscr A}u_t}^{\b_{n+1}-\frac{2}{n+2}}
\norm{A^p \mathscr A^{\frac{1}{2}} e^{\t \mathscr A}u_t}^{\frac{2}{n+2}}\right).
\end{align}

For the sake of notations, we denote
\begin{align} \label{def-gamma}
&\g_k := e^{\b_k}\norm{u}_{H^{p+1}}^{\b_k}, \;\;k=0, 1, \ldots, n ,\notag\\
&\g_{n+1} := e^{\b_{n+1}}\norm{u_t}_{H^{p}}^{\b_{n+1}},
\end{align}
and
\begin{align} \label{def-delta}
&\d_k:=(1+\norm{A^{p+1}e^{\t \mathscr A}u})^{\b_k-\frac{2}{n+2}},\;\;k=0, 1, \ldots, n,\notag\\
&\d_{n+1}:=(1+\norm{A^{p}e^{\t \mathscr A} u_t})^{\b_{n+1}-\frac{2}{n+2}}.
\end{align}
We remark that $\b_k$, $k=0,1,\ldots, n+1$ can be zero in (\ref{def-gamma})-(\ref{def-delta}).

Furthermore, we let $\a=(\a_0, \a_1,\ldots,\a_{n+1})$, where $\a_k=0$ or $1$, for all $k=0, 1,\ldots,n+1$. Also denote
$|\a|:=\sum_{k=0}^{n+1} \a_k$.

By (\ref{e-24})-(\ref{def-delta}) one has
\begin{align*}
&\norm{A^p e^{\t \mathscr A}u}^{\b_0}\prod_{k=1}^n \norm{A^p e^{\t \mathscr A}u_{x_k}}^{\b_k}\norm{A^p e^{\t \mathscr A}u_{t}}^{\b_{n+1}} \notag\\
&\leq 2^{|\b|-(n+2)}\sum_{\a} \left(
\Big[\t\Big(\norm{A^p \mathscr A^{\frac{1}{2}}e^{\t \mathscr A}u_t}^2+\norm{A^{p+1} \mathscr A^{\frac{1}{2}} e^{\t \mathscr A}u}^2\Big)\Big]^{\frac{|\a|}{n+2}}
\prod_{k=0}^{n+1}\gamma_k^{1-\a_k} \d_k^{\a_k} \right)\notag\\
&\leq 2^{|\b|-(n+2)}\prod_{k=0}^{n+1}\g_k \notag\\
&+2^{|\b|-(n+2)}\sum_{\a\not = \vec 0} \left(1+
\t\Big(\norm{A^p \mathscr A^{\frac{1}{2}} e^{\t \mathscr A}u_t}^2+\norm{A^{p+1} \mathscr A^{\frac{1}{2}} e^{\t \mathscr A}u}^2\Big)
\prod_{k=0}^{n+1}(\gamma_k^{1-\a_k} \d_k^{\a_k})^{\frac{n+2}{|\a|}} \right),
\end{align*}
for any $(\b_0,\ldots,\b_{n+1})\in \naturals^{n+2}_0$.
It follows that
\begin{align} \label{e-27}
&\tilde C g\left(t,C_0\norm{A^p e^{\t \mathscr A}u},C_0\norm{A^p e^{\t \mathscr A}u_{x_1}},
\cdots,C_0\norm{A^p e^{\t \mathscr A}u_{x_n}},C_0\norm{A^p e^{\t \mathscr A}u_{t}}\right)\notag\\
&\leq \tilde C\sum_{\b\in \naturals^{n+2}_0}b_{\b}(t) C_0^{|\b|} \Big(2^{|\b|}+2^{|\b|-(n+2)}\prod_{k=0}^{n+1}\g_k\Big) \notag\\
&\hspace{0.5 in}+\t \tilde C \sum_{\b\in \naturals^{n+2}_0}\left(b_{\b}(t)C_0^{|\b|} 2^{|\b|-(n+2)}
\sum_{\a\not = \vec 0}
\prod_{k=0}^{n+1}(\gamma_k^{1-\a_k} \d_k^{\a_k})^{\frac{n+2}{|\a|}} \right) \notag\\
&\hspace{1 in}\times
\Big(\norm{A^p \mathscr A^{\frac{1}{2}} e^{\t \mathscr A}u_t}^2+\norm{A^{p+1} \mathscr A^{\frac{1}{2}} e^{\t \mathscr A}u}^2\Big).
\end{align}

To estimate the right-hand side of (\ref{e-27}) we notice
\begin{align} \label{kappa}
&\tilde C\sum_{\b\in \naturals^{n+2}_0}b_{\b}(t) C_0^{|\b|} \Big(2^{|\b|}+2^{|\b|-(n+2)}\prod_{k=0}^{n+1}\g_k\Big)\notag\\
&=\tilde C g(t,2C_0,\ldots,2C_0)\notag\\
&+\frac{\tilde C}{2^{n+2}}g(t,2 e C_0\norm{u(t)}_{H^{p+1}},\ldots,2 e C_0\norm{u(t)}_{H^{p+1}}, 2 e C_0 \norm{u_t(t)}_{H^p})\notag\\
&:=\kappa (t).
\end{align}
Also
\begin{align} \label{psi}
&\tilde C \sum_{\b\in \naturals^{n+2}_0}\left(b_{\b}(t)C_0^{|\b|} 2^{|\b|-(n+2)}
\sum_{\a\not = \vec 0}
\prod_{k=0}^{n+1}(\gamma_k^{1-\a_k} \d_k^{\a_k})^{\frac{n+2}{|\a|}} \right) \notag\\
&\leq \tilde C\sum_{\a\not = \vec 0} \sum_{\b\in \naturals^{n+2}_0}\left(b_{\b}(t)C_0^{|\b|} 2^{|\b|-(n+2)}
\prod_{k=0}^{n+1}(\gamma_k^{1-\a_k} \d_k^{\a_k})^{\frac{n+2}{|\a|}} \right)  \notag\\
& \leq 2\tilde C g\Big(t,2 eC_0\left[1+\norm{u}_{H^{p+1}}+\norm{A^{p+1} e^{\t \mathscr A}u}\right]^{n+2},
\ldots, \notag\\
&2 eC_0\left[1+\norm{u}_{H^{p+1}}+\norm{A^{p+1} e^{\t \mathscr A}u}\right]^{n+2},
2 eC_0\left[1+\norm{u_t}_{H^{p}}+\norm{A^{p} e^{\t \mathscr A}u_t}\right]^{n+2}\Big)\notag\\
&:=\psi\left(t,\norm{u}_{H^{p+1}}, \norm{u_t}_{H^p},\norm{A^{p+1}e^{\t \mathscr A}u},\norm{A^p e^{\t \mathscr A}u_t}\right).
\end{align}

It follows from (\ref{e-27})-(\ref{psi}) that
\begin{align} \label{e-28}
&\tilde C g\left(t,C_0\norm{A^p e^{\t \mathscr A}u},C_0\norm{A^p e^{\t \mathscr A}u_{x_1}},
\cdots,C_0\norm{A^p e^{\t \mathscr A}u_{x_n}},C_0\norm{A^p e^{\t \mathscr A}u_{t}}\right) \notag\\
&\leq \kappa(t)+\t(t)\psi\left(t,\norm{u}_{H^{p+1}}, \norm{u_t}_{H^p},\norm{A^{p+1}e^{\t \mathscr A}u},\norm{A^p e^{\t \mathscr A}u_t}\right)\notag\\
&\hspace{1 in} \times \Big(\norm{A^p \mathscr A^{\frac{1}{2}} e^{\t \mathscr A}u_t}^2+\norm{A^{p+1} \mathscr A^{\frac{1}{2}} e^{\t \mathscr A}u}^2\Big).
\end{align}

If we set
\begin{align} \label{def-Y}
Y(t):=\left(\norm{A^{p} e^{\t(t) \mathscr A}u_t(t)}^2+\norm{A^{p+1}e^{\t(t) \mathscr A}u(t)}^2\right)^{\frac{1}{2}},
\end{align}
then by (\ref{eqn-1}), (\ref{e-22}) and (\ref{e-28}) we arrive at
\begin{align*}
Y(t) Y'(t)
\leq \kappa (t) Y(t)&+\left[\t'(t)+\t(t)Y(t) \psi\left(t,\norm{u}_{H^{p+1}}, \norm{u_t}_{H^p},Y(t),Y(t)\right)\right]\notag\\
&\times \Big(\norm{A^p \mathscr A^{\frac{1}{2}} e^{\t \mathscr A}u_t}^2+\norm{A^{p+1} \mathscr A^{\frac{1}{2}} e^{\t \mathscr A}u}^2\Big).
\end{align*}

Now we define
\begin{align} \label{def-eta}
\eta(t):=\psi\left(t,\norm{u}_{H^{p+1}}, \norm{u_t}_{H^p}, Y_0+\int_0^t \kappa(s) ds, Y_0+\int_0^t \kappa(s) ds\right)\left(Y_0+\int_0^t \kappa(s) ds\right),
\end{align}
where $\kappa$, $\psi$ are defined in (\ref{kappa}) and (\ref{psi}) respectively,
and $Y_0=(\norm{A^p e^{\t_0 \mathscr A}u_1}^2+\norm{A^{p+1}e^{\t_0 \mathscr A}u_0}^2 )^{\frac{1}{2}}$.
Thus, if $\t(t)$ solves the differential equation
\begin{align} \label{ODE-2}
\t'(t)+\t(t) \eta(t)=0 \text{\;\;with\;\;} \t(0)=\t_0>0,
\end{align}
then analog to the proof of Theorem \ref{thm-1}, we may conclude that $Y(t)$ is finite for all $t\in [0,T]$, i.e.,
$u(t)\in \mathcal D(A^{p+1}e^{\t(t) \mathscr A})$ and $u_t(t)\in \mathcal D(A^p e^{\t(t) \mathscr A})$ for all $t\in [0,T]$.
\end{proof}

\begin{remark}
The results in this section are also valid for general Gevrey-Sobolev classes
$\mathcal D(A^p e^{\t \mathscr A^{1/s}})$, $s\geq 1$, with its norm (\ref{G-norm}). Note, functions in $\mathcal D(A^p e^{\t \mathscr A^{1/s}})$, $s\geq 1$, have Gevrey class regularity of order $s$, in its first $m$ spatial variables, $m\leq n$. (One may recall the definition of the operator $\mathscr A$ in the Introduction.) In fact, we can follow the proof of the above results line by line to show if the initial data $(u_0,u_1)$ are in the spaces $\mathcal D(A^{p+1} e^{\sigma \mathscr A^{1/s}})$ and $\mathcal D(A^{p} e^{\sigma \mathscr A^{1/s}})$, respectively, where $p>n/2$, $s\geq 1$, then the solution $(u,u_t)$ belong to $\mathcal D(A^{p+1} e^{\t(t) \mathscr A^{1/s}})$ and $\mathcal D(A^{p} e^{\t(t) A^{1/s}})$, respectively, with $\t(t)$ specified in the above theorems. In short, our results are equally valid for any Gevrey class $s\geq 1$, provided the initial data are there.
\end{remark}

\bigskip
\section{$\ell^1$ estimate - the case of Wiener algebra} \label{sec-l-1}
In this section we employ a different Gevrey class of real analytic functions, which is based on the space of functions with summable Fourier series (Wiener algebra), to study the analytic regularity of solutions to nonlinear wave equations.
This type of Gevrey classes of real analytic functions was introduced in \cite{Oliver-Titi-01}.
In Section \ref{sec-compare} one will see, as it was also demonstrated in \cite{Oliver-Titi-01}, that such Gevrey class has its advantage of evaluating the radius of analyticity of solutions.

Let $u\in L^1({\mathbb{T}^n})$ with its Fourier series  $\sum_{j\in \ints^n}u_j e^{ij\cdot x}$. Then the $\ell^1$ norm of its Fourier transform is given by
$\norm{\hat u}_{\ell^1}:=\sum_{j\in \ints^n}|u_j|.$ This norm defines a Banach algebra, which is called \emph{Wiener algebra}, in the classic harmonic analysis.


The Gevrey norm based on the Wiener algebra is defined by
\begin{align} \label{Gevrey}
\norm{\hat u}_{G_{\t}(\ell^1)}:=\sum_{j\in \ints^n}e^{\t|j'|}|u_j|
\end{align}
where $j'$ stands for the first $m$ components of $j=(j_1,\ldots,j_n)\in \ints^n$, i.e., $j'=(j_1,\ldots,j_m)$ for some $m\leq n$. Furthermore if $u\in L^1(\mathbb T^n)$ such that $\norm{\hat u}_{G_{\t}(\ell^1)}$ is finite, we say that $\hat u$ belongs to $G_{\t}(\ell^1)$.

For clarity purposes we demonstrate the estimate for the wave equation (\ref{wave-eqn-1}) with the nonlinearity $f(t,x,u)$, and briefly discuss the general system (\ref{wave-eqn}) in Remark \ref{remark-2}.

The following theorem is concerned with Gevrey regularity of solutions with initial data in the Gevrey class $G_{\sigma}(\ell^1)$.
For the sake of comparing different estimates in the next section, we assume the existence of the same type of solutions as in Theorem \ref{thm-1}.

\begin{theorem} \label{thm-3}
Assume $f(t,x,u)$ satisfies Assumption \ref{ass-1}. Let $u_0\in H^{p}(\mathbb T^n)$, $u_1\in H^{p-1}(\mathbb T^n)$, $p>\frac{n}{2}$, and $\widehat {A u_0}$, $\widehat {u_1}\in G_{\sigma}(\ell^1)$.
Also suppose the initial value problem (\ref{wave-eqn-1})
has a unique solution $u\in C([0,T];H^p(\mathbb T^n))$ with $u_t\in C([0,T];H^{p-1}(\mathbb T^n))$, in the sense of Definition \ref{solution}. Then, $\widehat {A u(t)}$ and $\widehat {u_t(t)}$ both belong to the Gevrey class $G_{\t(t)}(\ell^1)$, for all $t\in [0,T]$, provided $\t(t)$ solves the differential equation
\begin{align*}
\t'(t)=-\t^2(t)\tilde h(t)  \text{\;\;with\;\;} \t(0)=\t_0:=\min\{\lambda,\sigma\},
\end{align*}
where $\tilde h(t)\geq 0$ for all $t\in [0,T]$, defined in (\ref{def-tilde-h}) below,
depends on $\norm{\widehat {u(t)}}_{\ell^1}$, $\norm{\widehat {A u_0}}_{G_{\t_0}(\ell^1)}$
and $\norm{\hat u_1}_{G_{\t_0}(\ell^1)}$.
\end{theorem}

\begin{proof}
The following calculations are formal, which can be justified rigorously by using the Galerkin method. Let the solution
$u(t)=\sum_{j\in \ints^n} u_j(t) e^{ij\cdot x}$,
where $u_j(t)$ are Fourier coefficients.
By assumption, the nonlinearity $f$ is in the form $f(t,x,u)=\sum_{j\in \ints^n}\left(\sum_{k=0}^{\infty} a_{jk}(t) u^k \right) e^{ij\cdot x}$. Then since $u$ is the solution of the equation $u_{tt}-\Delta u+u+f(t,x,u)=0$, we obtain for all $j\in \ints^n$,
\begin{align*}
u_j''(t) +(1+|j|^2) u_j(t)+a_{j0}(t)
+\sum_{l\in \ints^n}\sum_{k=1}^{\infty} \left(a_{lk}(t) \sum_{m_1+\cdots+m_k=j-l}u_{m_1}(t)\cdots u_{m_k}(t)\right)=0
\end{align*}
where $m_1,\ldots,m_k \in \ints^n$.

Thus
\begin{align*}
\begin{cases}
u_j''\bar u_j'+(1+|j|^2) u_j\bar u_j'+\bar u_j'\Big(a_{j0}+\sum_{l\in \ints^n}\sum_{k=1}^{\infty} a_{lk}\sum_{m_1+\cdots+m_k=j-l}u_{m_1}\cdots u_{m_k}\Big)=0\\
\bar u_j'' u_j'+(1+|j|^2) \bar u_j u_j'+u_j'\Big(a_{j0}+\sum_{l\in \ints^n} \sum_{k=1}^{\infty} a_{lk}\sum_{m_1+\cdots+m_k=j-l}\bar u_{m_1}\cdots \bar u_{m_k}\Big)=0.
\end{cases}
\end{align*}
Adding these two identities yields
\begin{align} \label{f-2}
&\frac{d}{dt}\left(|u_j'|^2+(1+|j|^2)|u_j|^2\right) \notag\\
&\leq 2|u_j'|\left[|a_{j0}|+\sum_{l\in \ints^n}\sum_{k=1}^{\infty}
\left( |a_{lk}|\sum_{m_1+\cdots+m_k=j-l}|u_{m_1}|\cdots |u_{m_k}|\right)\right].
\end{align}

If we denote
\begin{align} \label{def-phi}
\varphi_j:=\left(|u_j'|^2+(1+|j|^2)|u_j|^2\right)^{\frac{1}{2}},
\end{align}
then (\ref{f-2}) implies
\begin{align*}
\varphi_j'\leq |a_{j0}|+\sum_{l\in \ints^n}\sum_{k=1}^{\infty} \left(|a_{lk}|\sum_{m_1+\cdots+m_k=j-l}|u_{m_1}|\cdots |u_{m_k}|\right).
\end{align*}
It follows that
\begin{align*}
&\frac{d}{dt}\left(e^{\t(t)|j'|}\varphi_j\right)
\leq \t'(t)|j'|e^{\t(t)|j'|}\varphi_j+e^{\t(t)|j'|}|a_{j0}|
\notag\\
&+e^{\t(t)|j'|}\sum_{l\in \ints^n}\sum_{k=1}^{\infty}
\left(|a_{lk}|\sum_{m_1+\cdots+m_k=j-l}|u_{m_1}|\cdots |u_{m_k}|\right).
\end{align*}
Now, a summation over all $j\in \ints^n$ gives
\begin{align} \label{f-6}
&\frac{d}{dt}\left(\sum_{j\in\ints^n}e^{\t|j'|}\varphi_j\right)
\leq \t'\left(\sum_{j\in \ints^n}|j'|e^{\t|j'|}\varphi_j\right)
+\sum_{j\in \ints^n} e^{\t(t)|j'|}|a_{j0}|
\notag\\
&+\sum_{j\in \ints^n} \left(e^{\t|j'|}\sum_{l\in \ints^n}\sum_{k=1}^{\infty} |a_{lk}|\sum_{m_1+\cdots+m_k=j-l}|u_{m_1}|\cdots |u_{m_k}|\right).
\end{align}
To evaluate the last term in (\ref{f-6}), we rearrange the order of summations, obtaining
\begin{align} \label{f-7}
&\sum_{j\in \ints^n} \left(e^{\t|j'|}\sum_{l\in \ints^n}\sum_{k=1}^{\infty} |a_{lk}|\sum_{m_1+\cdots+m_k=j-l}|u_{m_1}|\cdots |u_{m_k}|\right) \notag\\
&=\sum_{k=1}^{\infty}\sum_{l\in \ints^n}\left( |a_{lk}|\sum_{j\in \ints^n}e^{\t |j'|}
\sum_{m_1+\cdots+m_k=j-l}|u_{m_1}|\cdots |u_{m_k}| \right)\notag\\
&\leq \sum_{k=1}^{\infty}\sum_{l\in \ints^n}\left( |a_{lk}|e^{\t |l'|}\sum_{j\in \ints^n}
\sum_{m_1+\cdots+m_k=j-l}(e^{\t|m_1'|}|u_{m_1}|)\cdots (e^{\t|m_k'|}|u_{m_k}|) \right)\notag\\
&\leq \sum_{k=1}^{\infty}\sum_{l\in \ints^n}\left[ |a_{lk}|e^{\lambda |l'|}\left(\sum_{j\in \ints^n} e^{\t|j'|}|u_j|\right)^k \right]\,,
\end{align}
provided $\t(t)\leq \lambda$, for all $t\geq 0$, where we have used the Young's inequality for convolutions.

Combining (\ref{f-6}) and (\ref{f-7}) gives
\begin{align} \label{f-3}
\frac{d}{dt}\left(\sum_{j\in \ints^n} e^{\t|j'|}\varphi_j\right)
&\leq \t'\left(\sum_{j\in \ints^n}|j'|e^{\t|j'|}\varphi_j\right)+\sum_{j\in \ints^n} e^{\t(t)|j'|}|a_{j0}| \notag\\
&+\sum_{k=1}^{\infty}\sum_{l\in \ints^n}\left[ |a_{lk}|e^{\lambda |l'|}\left(\sum_{j\in \ints^n} e^{\t|j'|}|u_j|\right)^k \right].
\end{align}

Like the proof of Theorem \ref{thm-1} we apply the elementary inequality $e^x\leq e+x^{\ell} e^x$ for $x\geq 0$, $\ell \geq 0$, and it follows
\begin{align} \label{f-4}
&\left(\sum_{j\in \ints^n}e^{\t|j'|}|u_j|\right)^k \notag\\
&\leq \left(\sum_{j\in \ints^n}e|u_j|+\sum_{j\in \ints^n}\t^{\frac{2}{k}}|j'|^{\frac{2}{k}}e^{\t|j'|}|u_j|\right)^k \notag\\
&\leq 2^{k-1}e^k\left(\sum_{j\in \ints^n}|u_j|\right)^k
+2^{k-1}\t^2\left(\sum_{j\in \ints^n}|j'|^{\frac{2}{k}}e^{\t|j'|}|u_j|\right)^k \notag\\
&\leq 2^{k-1}e^k\left(\sum_{j\in \ints^n}|u_j|\right)^k+
2^{k-1}\t^2\left(\sum_{j\in \ints^n}e^{\t|j'|}|u_j|\right)^{k-1}
\left(\sum_{j\in \ints^n}|j'|^{2}e^{\t|j'|}|u_j|\right) \,,
\end{align}
where we have used the discrete H\"older's inequality.

A combination of (\ref{f-3}) and (\ref{f-4}) yields
\begin{align} \label{f-5}
&\frac{d}{dt}\left(\sum_{j\in \ints^n} e^{\t|j'|}\varphi_j\right) \notag\\
&\leq \frac{1}{2}\sum_{k=1}^{\infty}\sum_{l\in \ints^n}|a_{lk}|e^{\lambda|l'|} \left(2e\sum_{j\in \ints^n}|u_j|\right)^k +\sum_{j\in \ints^n} e^{\t(t)|j'|}|a_{j0}|\notag\\
&+\left[\t'+\t^2 \sum_{k=1}^{\infty}\sum_{l\in \ints^n} |a_{lk}|e^{\lambda|l'|}\left(2\sum_{j\in \ints^n}e^{\t|j'|}|u_j|\right)^{k-1}\right]
\left(\sum_{j\in \ints^n}|j'| e^{\t|j'|}\varphi_j\right).
\end{align}

If we define
\begin{align} \label{def-y-2}
y(t)=\sum_{j\in \ints^n} e^{\t(t)|j'|}\varphi_j(t),
\end{align}
the estimate (\ref{f-5}) is reduced to
\begin{align*}
y'(t) \leq g(t,2e\norm{\widehat {u(t)}}_{\ell^1})
+[\t'+\t^2g(t,2y(t)+1)]\left(\sum_{j\in \ints^n}|j'|e^{\t|j'|}\varphi_j\right).
\end{align*}
Therefore, similar to Theorem \ref{thm-1}, we set
\begin{align} \label{def-tilde-h}
\tilde h(t)=g\left(t,2y(0)+2\int_0^t g(s,2e\norm{\widehat {u(s)}}_{\ell^1})ds+1\right),
\end{align}
and let $\t(t)$ solve the differential equation
\begin{align*}
\t'(t)+\t^2(t)\tilde h(t)=0  \text{\;\;with\;\;} \t(0)=\t_0,
\end{align*}
then it can be shown that
\begin{align} \label{f-8}
y(t)\leq y(0)+\int_0^t g(s,2e\norm{\widehat{u(s)}}_{\ell^1})ds.
\end{align}
Note for $p>\frac{n}{2}$, $H^p$ is imbedded in the Wiener algebra, which consists  of all the functions whose Fourier transform is in  $\ell^1$. Since the solution $u\in C([0,T];H^{p})$ and $\widehat {Au_0}$, $\widehat {u_1}\in G_{\t_0}(\ell^1)$, it guarantees that the right-hand side of (\ref{f-8}) is finite for all $t\in [0,T]$.

By (\ref{def-phi}) and (\ref{def-y-2}), it follows
\begin{align} \label{f-9}
\sum_{j\in \ints^n} e^{\t|j'|}|u_j'|+\sum_{j\in \ints^n} e^{\t|j'|}\sqrt{1+|j|^2}|u_j|\leq 2y(t),
\end{align}
and since $Au=\sum_{j\in \ints^n}u_j \sqrt{1+|j|^2} e^{ij\cdot x}$,
we obtain from (\ref{f-8})-(\ref{f-9}) that
$\widehat {A u(t)}$ and $\widehat {u_t(t)}$ both belong to the Gevrey class $G_{\t(t)}(\ell^1)$ for all $t\in [0,T]$.
\end{proof}

\begin{remark} \label{remark-2}
For the wave equation (\ref{wave-eqn}), which features the general nonlinearity $f(t,x,u,\nabla u,u_t)$ satisfying Assumption \ref{ass-2}, we can still study the regularity of its solution by employing the Gevrey class $G_{\t}(\ell^1)$ and estimate the radius of analyticity of the solution. Similar to Theorem \ref{thm-3}, we conclude if $u_0\in H^{p+1}(\mathbb T^n)$, $u_1\in H^p(\mathbb T^n)$, $p>\frac{n}{2}$, and $\widehat {Au_0}$, $\widehat {u_1} \in G_{\sigma}(\ell^1)$, and the initial value problem (\ref{wave-eqn}) has a unique solution
$u\in C([0,T];H^{p+1}(\mathbb T^n))$ with $u_t\in C([0,T];H^{p}(\mathbb T^n))$,
then $\widehat {A u(t)}$ and $\widehat {u_t(t)}$ both belong to the Gevrey class $G_{\t(t)}(\ell^1)$ for all $t\in [0,T]$,
provided $\t(t)$ solves the differential equation
\begin{align*}
\t'(t)=-\t(t)\tilde \eta(t)  \text{\;\;with\;\;} \t(0)=\t_0=\min\{\lambda,\sigma\},
\end{align*}
where the function $\tilde \eta(t)>0$, for all $t\in [0,T]$, depending on $\norm{\widehat {Au_0}}_{G_{\t_0}(\ell^1)}$,
$\norm{\widehat {u_1}}_{G_{\t_0}(\ell^1)}$, $\norm{\widehat {Au(t)}}_{\ell^1}$ and $\norm{\widehat {u_t(t)}}_{\ell^1}$.
The proof of this result combines techniques from Theorems \ref{thm-2} and \ref{thm-3}, and we omit the details.
\end{remark}

\bigskip

\section{Comparison of the $L^2$ and $\ell^1$ estimates} \label{sec-compare}
In the previous sections we use two different Gevrey classes to investigate the analytic regularity of solutions to nonlinear wave equations. In order to compare these estimates we consider a nonlinear Klein-Gordon equation:
\begin{align} \label{Klein}
\begin{cases}
\Box u+u\pm u^k=0, \;\; k\geq 2; \\
u(0)=u_0, \;\; u_t(0)=u_1,
\end{cases}
\end{align}
with the periodic boundary condition on $u$. Under this scenario we can carry out more accurate calculations due to the relatively simple structure of the nonlinearity in (\ref{Klein}).
Our purpose is to give an evidence, as it was done in \cite{Oliver-Titi-01}, to demonstrate that the $\ell^1$ estimate (Wiener algebra approach) can be more precise than the $L^2$ estimate for evaluating the radius of analyticity of solutions.

First we employ the  Gevrey class $G_{\t}(\ell^1)$ to study the regularity of the solution to (\ref{Klein}). In fact we have the following result.
\begin{proposition} \label{prop-2}
Let $u_0\in H^p$, $u_1\in H^{p-1}$, $p>\frac{n}{2}$, and $\widehat {A u_0}$, $\widehat {u_1}\in G_{\sigma}(\ell^1)$. Suppose the initial value problem (\ref{Klein}) has a unique solution $u\in C([0,T];H^{p})$ with $u_t\in C([0,T];H^{p-1})$, in the sense of Definition \ref{solution}. Then, $\widehat {A u(t)}$ and $\widehat {u_t(t)}$ both belong to the Gevrey class $G_{\t(t)}(\ell^1)$, for all $t\in [0,T]$, provided $\t(t)$ solves the differential equation
\begin{align} \label{radius}
\t'(t)=-\t^{k+1}(t)h_1(t), \text{\;\;with\;\;} \t(0)=\sigma,
\end{align}
where $h_1(t)\geq 0$ for all $t\in [0,T]$, defined in (\ref{def-h-1}), depends on $\norm{\widehat {u(t)}}_{\ell^1}$, $\norm{\widehat {A u_0}}_{G_{\sigma}(\ell^1)}$
and $\norm{\widehat {u_1}}_{G_{\sigma}(\ell^1)}$.
\end{proposition}

\begin{proof}
Following the estimate in the proof of Theorem \ref{thm-3}, we reach
\begin{align} \label{k-13}
\frac{d}{dt}\left(\sum_{j\in \ints^n} e^{\t(t) |j'|}\varphi_j(t)\right)
&\leq \t'\left(\sum_{j\in \ints^n}|j'|e^{\t(t) |j'|}\varphi_j(t)\right)+
\left(\sum_{j\in \ints^n} e^{\t(t) |j'|}|u_j(t)|\right)^k
\end{align}
where
$\varphi_j=\left[|u_j'|^2+(1+|j|^2)|u_j|^2\right]^{\frac{1}{2}}$.

Next, we evaluate the last term in (\ref{k-13}). The estimate will be slightly different from (\ref{f-4}) in the proof of Theorem \ref{thm-3}. Indeed, in the estimate (\ref{f-4}) we require the exponent of $\t$, on the right-hand side of the inequality, to remain the same for all $k\geq 1$, i.e., we demand the term $\t^2$ appears in the estimate, no matter what the value of $k$ is. This is important when we calculate the summation over all $k\geq 1$. But in the Klein-Gordon equation (\ref{Klein}) the nonlinearity is a monomial $u^k$, which provides us more freedom to perform the evaluation. By using the inequality $e^x\leq e+x^{\ell} e^x$ for all $x\geq 0$, $\ell \geq 0$, we deduce
\begin{align} \label{k-14}
&\left(\sum_{j\in \ints^n}e^{\t|j'|}|u_j|\right)^k \notag\\
&\leq \left(\sum_{j\in \ints^n}e|u_j|+\sum_{j\in \ints^n}\t^{\frac{k+1}{k}}|j'|^{\frac{k+1}{k}}e^{\t|j'|}|u_j|\right)^k \notag\\
&\leq 2^{k-1}e^k \left(\sum_{j\in \ints^n}|u_j|\right)^k
+2^{k-1}\t^{k+1}\left(\sum_{j\in \ints^n}|j'|^{\frac{k+1}{k}}e^{\t|j'|}|u_j|\right)^k \notag\\
&\leq 2^{k-1}e^k\left(\sum_{j\in \ints^n}|u_j|\right)^k+
2^{k-1}\t^{k+1} \left(\sum_{j\in \ints^n} |j'|e^{\t|j'|}|u_j|\right)^{k-1}
\left(\sum_{j\in \ints^n} |j'|^{2}e^{\t|j'|}|u_j|\right),
\end{align}
where we have used the discrete H\"older's inequality.

If we let
$y(t)=\sum_{j\in \ints^n} e^{\t(t)|j'|}\varphi_j$,
then its follows from (\ref{k-13}) and (\ref{k-14}) that
\begin{align*}
y'(t) \leq 2^{k-1} e^k \norm{\widehat {u(t)}}_{\ell^1}^k+
[\t'+2^{k-1} \t^{k+1} y^{k-1}(t)]\left(\sum_{j\in \ints^n}|j'|e^{\t|j'|}\varphi_j\right).
\end{align*}
Now, we define
\begin{align} \label{def-h-1}
h_1(t)&=2^{k-1} \left(y(0)+2^{k-1}  e^k\int_0^t \norm{\widehat {u(s)}}^k_{\ell^1} ds\right)^{k-1}\notag\\
&=\left(2y(0)+2^k e^k\int_0^t \norm{\widehat {u(s)}}^k_{\ell^1} ds\right)^{k-1}.
\end{align}
Therefore if $\t(t)$ is the solution of the differential equation
\begin{align*}
\t'(t)+\t^{k+1}(t) h_1(t)=0 \text{\;\;with\;\;} \t(0)=\sigma,
\end{align*}
then
$y(t)\leq y(0)+2^{k-1}  e^k\int_0^t \norm{\widehat {u(s)}}^k_{\ell^1} ds$ for all $t\in [0,T]$, i.e.,
$\widehat {A u(t)}$ and $\widehat {u_t(t)}$ both belong to the Gevrey class $G_{\t(t)}(\ell^1)$ for all $t\in [0,T]$.
\end{proof}

In order to do the comparison, we study the same problem (\ref{Klein}) by employing the Gevrey-Sobolev class
$\mathcal D(A^p e^{\t \mathscr A})$ defined in (\ref{G-norm'}).

\begin{proposition} \label{prop-3}
Let $u_0\in \mathcal D(A^{p+1}e^{\sigma \mathscr A})$ and $u_1\in \mathcal D(A^p e^{\sigma \mathscr A})$ for $p>\frac{n}{2}$ and
$\sigma>0$.
Assume the initial-value problem (\ref{Klein}) has a unique solution $u\in C([0,T];H^{p})$ with $u_t\in C([0,T];H^{p-1})$, in the sense of Definition \ref{solution}.
Then $u(t)\in \mathcal D(A^{p+1} e^{\t(t)\mathscr A})$ and $u_t(t)\in \mathcal D(A^p e^{\t(t) \mathscr A})$, for all $t\in [0,T]$,
provided $\t(t)$ satisfies the equation
\begin{align}  \label{radius'}
\t'(t)=-\t^{k+1}(t)h_2(t) \text{\;\;with\;\;} \t(0)=\sigma,
\end{align}
where
\begin{align} \label{def-h-2}
h_2(t)&=\left(C_0 Y_0+\frac{1}{2}C_0^{k} (e\sqrt{2})^k \int_0^t \norm{u(s)}^k_{H^p}ds\right)^{k-1},
\end{align}
where $Y_0=(\norm{A^p e^{\s \mathscr A}u_1}^2+\norm{A^{p+1}e^{\s \mathscr A}u_0}^2 )^{\frac{1}{2}}$.
\end{proposition}

\begin{proof}
Follow the estimate in the proof of Theorem \ref{thm-1} and adopt ideas from the calculations in Proposition \ref{prop-2}. We omit the details of the proof.
\end{proof}

Here, we provide an example of applications of the above proposition. Consider the Klein-Gordon equation with cubic nonlinearity:
\begin{align} \label{cubic-wave}
\begin{cases}
&\Box u+u+u^3=0, \;\;\;(t,x_1,x_2,x_3)\in \reals \times \reals^3, \\
&u(0)=u_0,\;\;u_t(0)=u_1.
\end{cases}
\end{align}
Notice, the energy of this equation is bounded and all solutions exist globally. In particular, the strong solution $(u,u_t)\in H^2 \times H^1$ exists globally with $u\in L^{\infty}(\reals^+,H^2)$ \cite{Brenner-79, Brenner-89}. Therefore, by Proposition \ref{prop-3} (also valid for the equation defined in the whole space $\reals^n$), we conclude, if the initial data $u_0$ and $u_1$ are both real analytic in the spatial variable $x_1$ (for instance), then the solution $(u,u_t)$ remains real analytic in $x_1$ for all time, with the radius of analyticity bounded below by $\t(t)$ of the asymptotic decay rate
\begin{align*}
\t(t) \thicksim \frac{1}{t}.
\end{align*}

\begin{remark}
Notice that the $L^2$ approach (Proposition \ref{prop-3})
requires $(u_0,u_1)\in H^{p+1}\times H^p$, $p>n/2$;
while the Wiener algebra approach (Proposition \ref{prop-2}) asks for less smoothness of the initial data: $(u_0,u_1)\in H^{p}\times H^{p-1}$, $p>n/2$.

Also, we find the equations (\ref{radius}) and (\ref{radius'}) are almost identical except the functions $h_1$ and $h_2$. Thus in order to compare the lower bounds of the radius of analyticity $\t(t)$ given by Propositions \ref{prop-2} and \ref{prop-3}, we shall compare the values of $h_1(t)$ and $h_2(t)$ for $t\in [0,T]$.

Indeed, we consider
\begin{align} \label{k-8}
\norm{\hat u}^2_{\ell^1}=\left(\sum_{j\in \ints^n}|u_j|\right)^2&\leq
\left(\sum_{j\in \ints^n}(1+|j|^2)^{-p}\right)\left(\sum_{j}(1+|j|^2)^{p} |u_j|^2\right)\notag\\
&=\left(\sum_{j\in \ints^n}(1+|j|^2)^{-p}\right)\norm{u}_{H^p}^2.
\end{align}

From Lemma \ref{lem-2} below we know that $C_0^2=2^{2p+1}\sum_{j\in \ints^n}(1+|j|^2)^{-p}$, i.e.,
\begin{align} \label{k-9}
\sum_{j\in \ints^n}(1+|j|^2)^{-p}=\frac{C_0^2}{2^{2p+1}}\leq \frac{C_0^2}{4},
\end{align}
for $p> n/2$.
It follows from (\ref{k-8}) and (\ref{k-9}) that
\begin{align} \label{k-10}
\int_0^t \norm{\widehat {u(s)}}^k_{\ell^1}ds
&\leq \left(\sum_{j\in \ints^n}(1+|j|^2)^{-p}\right)^{\frac{k}{2}}
\int_0^t \norm{u(s)}_{H^p}^k ds\notag\\
&\leq C_0^k 2^{-k} \int_0^t \norm{u(s)}_{H^p}^k ds.
\end{align}
Finally, we notice that
\begin{align} \label{k-11}
y(0)&=\sum_{j\in \ints^n} e^{\s |j'|}\varphi_j(0) \notag\\
&\leq \left(\sum_{j\in \ints^n}(1+|j|^2)^{-p}\right)^{\frac{1}{2}}
\left(\sum_{j\in \ints^n}e^{2\s|j'|}(1+|j|^2)^{p} \varphi_j^2(0)\right)^{\frac{1}{2}}\notag\\
&\leq \frac{C_0}{2}\left(\norm{A^p e^{\s \mathscr A}u_1}^2+\norm{A^{p+1} e^{\s \mathscr A}u_0}^2\right)^{\frac{1}{2}}
=\frac{C_0}{2} Y_0.
\end{align}

By substituting (\ref{k-10}) and (\ref{k-11}) into (\ref{def-h-1}) we obtain
\begin{align} \label{k-12}
h_1(t)\leq \left(C_0 Y_0+e^k C_0^k \int_0^t \norm{u(s)}_{H^p}^k ds\right)^{k-1}\leq h_2(t)
\end{align}
for all $t\in [0,T]$, if $k\geq 2$. It follows that the $\ell^1$ estimate provides larger radius of analyticity than the $L^2$ estimate does. Indeed, the reason of this fact is simply the imbedding
\begin{align}  \label{k-16}
\norm{\hat u}_{\ell^1}\leq C(p)\norm{u}_{H^p} \text{\;\;if\;\;} p>\frac{n}{2}.
\end{align}
As pointed out in \cite{Oliver-Titi-01}, the inequality (\ref{k-16}) becomes increasingly unsaturated - it has a large gap between its left and right hand sides - when $u$ is dominated by contributions from high wavenumbers. For example, if we set $u(x)=e^{imx}$, then the right hand side of (\ref{k-16}) increases with $m$, while the left hand side remains constant.
\end{remark}

\begin{remark} \label{remark-3}
Another way to see the advantage of $\ell^1$ estimate is to study the scaling behavior of the radius of analyticity with respect to physical parameters. In fact, for the equation
\begin{align} \label{k-19}
u_{tt}-\nu u_{xx}+\l u-u^3=0,
\end{align}
\cite{Liu-01} gives an explicit real analytic periodic solution
\begin{align} \label{k-18}
u=\sqrt{\frac{2m^2 \l}{1+m^2}}\sn\left(\sqrt{\frac{\l}{(1+m^2)(c^2-\nu )}}(x-ct)\right)
\end{align}
where $\sn$ is a Jacobi elliptic function with the modulus $m$, and $c^2>\nu>0$.
Notice that, the ODE which describes the steady states of (\ref{k-19}) was considered in \cite{Oliver-Titi-01} and it was shown that the estimates on the radius of analyticity obtained by the usual Gevrey class approach do not scale optimally as a function of the physical parameters, and in order to remedy it, the authors gave a modified definition of the Gevrey class based on the Wiener algebra, which was shown to yield a sharp scaling behavior of the estimates on the radius of analyticity. Their discovery can be verified here as well, for the Klein-Gordon (\ref{k-19}). For instance,
 it is easy to see from the explicit solution (\ref{k-18}) that, as $\nu \rightarrow 0$ the radius $\rho$ of analyticity of $u$ has the same asymptotic behavior with $C \nu^{\frac{1}{2}}$, that is, $\rho \sim C \nu^{\frac{1}{2}}$, where $C$ is a constant. By carrying out similar evaluations as in Propositions \ref{prop-2} and \ref{prop-3} we find that $\t(t)$, a lower bound of the radius of analyticity of the solution to (\ref{k-19}), obtained by the Gevrey estimate based on Wiener algebra, scales optimally as $\nu\rightarrow 0$; while the usual $L^2$ Gevrey estimate shows $\t \sim C \nu^{\frac{1+p}{2}}$, which is lack of sharpness since $p>\frac{1}{2}$.
But we omit the detail of calculations.
\end{remark}

\begin{remark} \label{remark-4}
Finally, we comment that, by using Fourier transforms instead of Fourier series, our results in the paper are also valid for nonlinear wave equations in the whole space $\reals^n$ or on the sphere $\mathbb{S}^{n-1}$. One refers to \cite{Oliver-Titi-00} for Gevrey estimates of Navier-Stokes equations in $\reals^3$, and to \cite{Cao-Rammaha-Titi1, Cao-Rammaha-Titi2} on $\mathbb{S}^{2}$, using spherical harmonics as a basis instead of the trigonometric functions in the periodic. Furthermore, the tools and results presented here can be extended in a straightforward manner to other  equations, such as the nonlinear Schr\"odinger equation with real analytic nonlinearity.

Recently, adopting the Gevrey class energy-like method, the authors of \cite{Guo-Titi-13} considered the analytic regularity of a non-dispersive Hamiltonian equation: the cubic Szeg\H o equation. By taking advantage of the uniform boundedness of the $\ell^1$ norm of the Fourier transform of the solution, the method based on the Wiener algebra provided a substantially better estimate (exponential decay) of the analyticity radius of the solution than the one (double exponential decay) obtained by the regular $L^2$ approach. Furthermore, the idea of working in the Wiener algebra to study spatial analyticity of solutions was recently applied to the incompressible Navier-Stokes system on $n$-torus \cite{Biswas-Jolly-Martinez-Titi}, where semigroup techniques were used.

Another  important example  is the following Cauchy problem for wave equations with exponential nonlinearities in the two-dimensional space:
\begin{align} \label{k-17}
\Box u+ u e^{u^2}=0 \text{\;\;on\;\;} \reals \times \reals^2.
\end{align}
Ibrahim et al.   \cite{Ibrahim-06} established    the global well-posedness (in time) of (\ref{k-17}), for  ($C^{\infty}(\reals^2)$) initial data of restricted size. Later,  Struwe \cite{Struwe-11, Struwe-11'} improved this result and  removed the restriction on the size of initial data. Since the nonlinearity $u e^{u^2}$ is analytic in $u$, the results  and tools presented in this paper are  applicable to (\ref{k-17}) with real analytic initial data in $\mathbb{R}^2$. Thus, by combining  our results  with the   global regularity result of Struwe \cite{Struwe-11, Struwe-11'}, one concludes that the solution of (\ref{k-17}) remains analytic for all $t\ge 0$, provided the initial data is real analytic in $\mathbb{R}^2$.

\end{remark}

\bigskip

\section{Appendix}
In the Appendix we prove some properties of the Gevrey classes used in the paper.
For more on Gevrey classes, see \cite{Ferrari-Titi-98, Foias-Temam-89, Levermore-Oliver-97, Oliver-Titi-00, Oliver-Titi-01}.
First we show the Gevrey classes $G_{\t}(\ell^1)$ and $\mathcal D(A^p e^{\t \mathscr A})$ correspond to functions which are analytic in certain arguments.
\begin{lemma} \label{lem-1}
Let $u(x)=\sum_{j\in \ints^n}u_j e^{ij\cdot x}$, where $x=(x_1, \ldots, x_n)\in \mathbb T^n$, such that
$\sum_{j\in \ints^n}|u_j|e^{\t|j'|}<\infty$ for all $\t\in (0,\sigma)$, where $j'=(j_1,\ldots,j_m)$,
$m\leq n$. Then $u$ is real analytic in the variable $(x_1,\ldots,x_m)\in \mathbb T^m$ with uniform radius of analyticity $\sigma$. That is, the function
$u(z_1,\ldots,z_m,x_{m+1},\ldots,x_n)=\sum_{j\in \ints^n}u_j
e^{i(\sum_{k=1}^m j_k z_k)} e^{i(\sum_{k=m+1}^n j_k x_k)}$
is analytic in the variables $z_1,\ldots,z_m$, where
$z_k=x_k+iy_k$, in the domain $x\in \mathbb T^n$, $\sum_{k=1}^m |y_k|^2<\sigma^2$.
\end{lemma}

\begin{proof}
Notice that the function $e^{i(\sum_{k=1}^m j_k z_k)} e^{i(\sum_{k=m+1}^n j_k x_k)}$ is entire in the variables $z_1,\ldots,z_m$. Thus we need to show that, the series
$\sum_{j\in \ints^n}u_j e^{i(\sum_{k=1}^m j_k z_k)} e^{i(\sum_{k=m+1}^n j_k x_k)}$ is convergent uniformly for all $x\in \mathbb T^n$, $\sum_{k=1}^m |y_k|^2\leq \t^2<\sigma^2$.
In fact,
\begin{align*}
\sum_{j\in \ints^n}|u_j| |e^{i(\sum_{k=1}^m j_k z_k)}| |e^{i(\sum_{k=m+1}^n j_k x_k)}|
\leq  \sum_{j\in \ints^n}|u_j| e^{|y||j'|}\leq \sum_{j\in \ints^n}|u_j| e^{\t|j'|}<\infty.
\end{align*}
\end{proof}

\begin{corollary} \label{cor-1}
Let $u\in \mathcal D(A^p e^{\t \mathscr A})$ for all $\t\in (0,\sigma)$.
If $p>n/2$, then $u$ is real analytic in the variables $x'=(x_1,\ldots,x_m)\in \mathbb T^m$
with uniform radius of analyticity $\sigma$.
\end{corollary}
\begin{proof}
By H\"older's inequality, we have
\begin{align*}
\sum_{j\in \ints^n} e^{\t|j'|}|u_j| \leq \left(\sum_{j\in \ints^n}\frac{1}{(1+|j|^2)^p}\right)^{\frac{1}{2}}\left(\sum_{j\in \ints^n} (1+|j|^2)^p e^{2\t|j'|}|u_j|^2\right)^{\frac{1}{2}}<\infty
\end{align*}
if $p>\frac{n}{2}$.
\end{proof}

\smallskip
The next result states the Gevrey-Sobolev class $\mathcal D(A^p e^{\t \mathscr A})$ is an algebra.
\begin{lemma} \label{lem-2}
If $u$ and $v$ are in the Gevrey-Sobolev class $\mathcal D(A^p e^{\t \mathscr A})$ with $p>n/2$, then their product $uv\in \mathcal D(A^p e^{\t \mathscr A})$ and
\begin{align*}
\norm{A^p e^{\t \mathscr A} (uv)}\leq C_0 \norm{A^p e^{\t \mathscr A} u} \norm{A^p e^{\t \mathscr A}v}
\end{align*}
where $C_0=2^p\sqrt{2\sum_{j\in \ints^n}(1+|j|^2)^{-p}}$.
\end{lemma}
\begin{proof}
Similar to Lemma 1 in \cite{Ferrari-Titi-98} with careful estimate of the constant $C_0$.
\end{proof}

\par\smallskip\noindent
\textbf{Acknowledgement\,:}   This paper is dedicated to Professor Neil Trudinger
on the occasion of his  70th birthday,  as token of friendship
and admiration for him as a teacher, and for his great contribution to research in partial
differential equations. This work was  supported in part by the Minerva Stiftung/Foundation, and  by the NSF
grants DMS-1009950, DMS-1109640 and DMS-1109645.



\begin{thebibliography}{99}
\bibitem{Alinhac-Metivier-84} S. Alinhac and G. Metivier, Propagation de l'analyticit\'e des solutions d'\'equations hyperboliques non-lin\'eaires, Inventiones Mathematicae 75 (1984), 189-204.

\bibitem{Bardos-Benachour} C. Bardos and S. Benachour, Domaine d'analyticit\'e des solutions de l'\'equation d'Euler dans un ouvert
de $\mathbb{R}^n$, Annal. Sc. Normale Sup. di Pisa, Volume d\'edi\'e \`a Jean Leray (1978),  507-547.

\bibitem{Biswas-Jolly-Martinez-Titi} A. Biswas, M. S. Jolly, V. Martinez, and E. S. Titi, Smallest scale estimates for the Navier-Stokes equations in the Wiener algebra, preprint (2013).

\bibitem{Biswas-Swanson} A. Biswas and D. Swanson, Gevrey regularity of solutions to the 3-D Navier-Stokes equations with weighted $l_p$ initial data,  Indiana University Mathematics Journal 56 (2007), 1157-1188.

\bibitem{Brenner-79} P. Brenner, On the existence of global smooth solutions for certain
semilinear hyperbolic equations, Mathematische Zeitschrift 167 (1979), 99-135.

\bibitem{Brenner-89} P. Brenner, On space-time means and strong global solutions of nonlinear hyperbolic equations, Mathematische Zeitschrift 201 (1989), 45-55.

\bibitem{Cao-Rammaha-Titi1}  C. Cao, M. A. Rammaha, and E. S. Titi,
Gevrey regularity for nonlinear analytic parabolic equations
on the sphere, Journal of Dynamics \& Differential Equations, 12 (2000), 411-433.

\bibitem{Cao-Rammaha-Titi2} C. Cao, M. A. Rammaha, and E. S. Titi,
 The Navier--Stokes equations on the rotating $2-D$ sphere:
Gevrey regularity and asymptotic degrees of freedom,
 Zeitschrift f{\"{u}}r Angewandte
Mathematik und Physik (ZAMP)   50  (1999), 341-360.

\bibitem{Doelman-Titi-93} A. Doelman and E. S. Titi, Regularity of solutions and the convergence of the Galerkin method in the Ginzburg-Landau equation,
Numerical Functional Analysis and Optimization 14 (1993), 299-321.

\bibitem{Ferrari-Titi-98} A. B. Ferrari and E. S. Titi, Gevrey regularity for nonlinear analytic parabolic equations, Communications in Partial Differential Equations 23 (1998), 1-16.

\bibitem{Foias-Temam-89} C. Foias and  R. Temam, Gevery class regularity for the solutions of the Navier-Stokes equations, Journal of Functional Analysis 87 (1989), 359-369.

\bibitem{Guo-Titi-13} Patrick G\'erard, Yanqiu Guo, and Edriss S. Titi, On the radius of analyticity of solutions to the cubic Szeg\H o equation, preprint (2013), \underline{arXiv:1303.6148}

\bibitem{Ibrahim-06} S. Ibrahim, M. Majdoub, and N. Masmoudi, Global solutions for a semilinear two-dimensional Klein-Gordon equation with exponential-type nonlinearity, Communications on Pure and Applied Mathematics 59 (2006), 1639-1658.

\bibitem{Kalantatov-Levant-Titi} V. K. Kalantarov, B. Levant, and E. S. Titi,
  Gevrey regularity of the global attractor of the 3D Navier-Stokes-Voight equations, Journal
of Nonlinear Science 19 (2009), 133-152.

\bibitem{Kreiss} H. O. Kreiss,  Fourier expansions of the Navier-Stokes
equations and their exponential decay rate, Analyse Math\'{e}matique
et Applications, Gauthier-Villars, Paris (1988), 245-262.

\bibitem{Kukavica-Vicol-09} I. Kukavica and V. Vicol, On the radius of analyticity of solutions to the three-dimensional Euler equations, Proceedings of the American Mathematical Society 137 (2009), 669-677.

\bibitem{Kuksin-Nadirashvili-12} S. Kuksin and N. Nadirashvili, Analyticity of solutions for quasilinear wave equations and other quasilinear systems, preprint (2012), arXiv:1205.5926.

\bibitem{Larios-Titi-10} A. Larios and E. S. Titi, On the higher-order global regularity of the inviscid Voigt-regularization of three-dimensional
hydrodynamic models, Discrete and Continuous Dynamical Systems Series B 14 (2010), 603-627.

\bibitem{Levermore-Oliver-97} C. D. Levermore and M. Oliver, Analyticity of solutions for a generalized Euler equation, Journal of Differential Equations 133 (1997), 321-339.

\bibitem{Liu-01} S. Liu, Z. Fu, S. Liu, and Q. Zhao, Jacobi elliptic function expansion method and periodic wave solutions of nonlinear wave equations, Physics Letter A 289 (2001), 69-74.

\bibitem{Nirenberg-72} L. Nirenberg, An abstract form of the nonlinear Cauchy-Kowalewski theorem, Journal of Differential Geometry 6 (1972) 561-576.

\bibitem{Oliver-Titi-00} M. Oliver and E. S. Titi, Remark on the rate of decay of higher order derivatives for solutions to the Navier-Stokes equations in $\reals^n$, Journal of Functional Analysis 172 (2000), 1-18.

\bibitem{Oliver-Titi-01} M. Oliver and E. S. Titi, On the domain of analyticity for solutions of second order analytic nonlinear differential equations,
Journal of Differential Equations 174 (2001), 55-74.

\bibitem{Ovsjannikov-71} L. V. Ovsjannikov, A nonlinear Cauchy problem in a scale of Banach spaces,
Dokl. Akad. Nauk SSSR, 200 (1971) 789-792; Soviet Math. Dokl. 12 (1971) 1497-1502.

\bibitem{Simon-87} J. Simon, Compact sets in the space $L^p(0,T;B)$, Annali di Matematica pura ed applicata (IV) Vol. CXLVI (1987), 65-96.

\bibitem{Struwe-11} M. Struwe, Global well-posedness of the Cauchy problem for a super-critical nonlinear wave equation in two space dimensions, Mathematische Annalen 350 (2011), 707-719.

\bibitem{Struwe-11'} M. Struwe, A `super-critical' nonlinear wave equation in 2 space dimensions, Milan Journal of Mathematics 79 (2011), 129-143.

\bibitem{Takac-96} P. Takac, P. Bollerman, A. Doelman, A. Van Harten, and E. S. Titi, Analyticity of essentially bounded solutions to semilinear parabolic systems and validity of the Ginzburg-Landau equation, SIAM Journal on Mathematical Analysis 27 (1996), 424-448.

\bibitem{Temam-83} R. Temam, Navier-Stokes equations and nonlinear functional analysis, CBMS-NSF regional conference series in applied mathematics, SIAM, 1983.

\end{thebibliography}
\end{document}